\documentclass[11pt,a4paper,notitlepage,oneside]{amsart}
\usepackage{etoolbox}
\usepackage{microtype} 
\usepackage{amsmath}
\usepackage{amsthm}
\usepackage{braket}
\usepackage{pdfpages}
\usepackage[psamsfonts]{amssymb}
\usepackage[T1]{fontenc}
\usepackage[english]{babel}
\usepackage[utf8]{inputenc}
\usepackage{enumitem}
\usepackage{mathtools}
\usepackage{caption,subfig} 
\usepackage{rotating} 
\usepackage{scalerel}
\usepackage{relsize} 
\usepackage{cancel} 
\usepackage{todonotes}

\usepackage[autostyle,italian=guillemets]{csquotes}

\allowdisplaybreaks 

\usepackage{pgf,tikz,pgfplots}
\usepackage{mathrsfs}
\usetikzlibrary{arrows}

\usepackage[hyphenbreaks]{breakurl}

\usepackage{fancyhdr}
\usepackage{tikz-cd}
\usepackage{amscd}

\usepackage[a4paper, bottom=3cm, left=3cm,right=3cm]{geometry}
\usepackage{hyperref}
\hypersetup{hidelinks} 

\usepackage{subfiles}
\numberwithin{equation}{section}

\usepackage[style=alphabetic,sorting=nyt,firstinits=true,maxnames=5,maxalphanames=5,backend=biber,doi=false, url=false, isbn=false]{biblatex}
{\theoremstyle{plain}\newtheorem{Def}{Definition}[section]}
\newtheorem{lem}[Def]{Lemma}
\newtheorem{teor}[Def]{Theorem}
\newtheorem{Cor}[Def]{Corollary}
\newtheorem{prop}[Def]{Proposition}
\newtheorem{conj}[Def]{Conjecture}
{\theoremstyle{remark} \newtheorem{Rem}[Def]{\bf Remark}}
{\theoremstyle{remark} \newtheorem{Exa}[Def]{Example}}
\newtheorem*{rem*}{Remark}

\newcommand{\sqH}{H_{\scriptscriptstyle(2)}}
\newcommand{\ZZ}{\mathbb{Z}}
\newcommand{\QQ}{\mathbb{Q}}

\newcommand{\RR}{\mathbb{R}}
\newcommand{\CC}{\mathbb{C}}

\newcommand{\HH}{\mathbb{H}}
\newcommand{\A}{\mathbb{A}}
\newcommand{\WW}{\mathbb{W}}

\DeclareMathOperator{\SL}{SL}

\DeclareMathOperator{\tr}{tr}

\newcommand{\vect}[1]{\boldsymbol{\mathrm{#1}}}

\def\XXint#1#2#3{{\setbox0=\hbox{$#1{#2#3}{\int}$}
     \vcenter{\hbox{$#2#3$}}\kern-.5\wd0}}

\newcommand{\Orth}{\mathrm{O}}
\newcommand{\Sp}{\operatorname{Sp}}

\newcommand{\newcol}{black} 

\title{%
     The Siegel--Weil formula in geometry and arithmetic
}
\author{Jan Bruinier}
\address{Jan Bruinier: Fachbereich Mathematik, Technische Universität Darmstadt, D-64289 Darmstadt,  Deutschland.}
\email{bruinier@mathematik.tu-darmstadt.de}
\author{Riccardo Zuffetti}
\address{Riccardo Zuffetti: Fachbereich Mathematik, Technische Universität Darmstadt, D-64289 Darmstadt,  Deutschland.}
\email{zuffetti@mathematik.tu-darmstadt.de}

\begin{document}

    \maketitle
    \thispagestyle{empty}

    \normalsize
    \setcounter{tocdepth}{1}


\begin{abstract}
The present paper is an extended version of the lecture notes of a course given by the first author at the summer school on \emph{Formulas of Siegel and Weil}  (Bielefeld, September 2025).
We survey three perspectives on the Siegel--Weil formula: classical, geometric, and arithmetic.
We first recall the Siegel--Weil formula for elliptic theta series arising from positive definite lattices. We next discuss the higher genus case for lattices of arbitrary signature from an adelic viewpoint.
After introducing orthogonal Shimura varieties, we present the geometric Siegel--Weil formula, in which the generating series of volumes of special cycles is shown to be an Eisenstein series.
We conclude with the (partly conjectural)  arithmetic Siegel--Weil formula, relating degrees of special cycles in arithmetic Chow groups to central derivatives of Eisenstein series.

\end{abstract}

\tableofcontents

\section{Introduction}

The Siegel--Weil formula is one of the most important identities relating 
the arithmetic of quadratic forms and automorphic forms. In its classical form, it says that a suitable average of theta series attached to positive definite lattices in a fixed genus is an Eisenstein series. Since the Fourier coefficients of theta series count representation numbers of quadratic forms, and since the Fourier coefficients of Eisenstein series can be often explicitly computed, this identity provides a bridge between the arithmetic of lattices and the theory of modular forms.

The purpose of the present paper is to give an introduction to the classical Siegel--Weil formula and explain its geometric and arithmetic extensions relating special cycles on orthogonal  Shimura varieties to Siegel modular forms.

We begin with the case of positive definite even lattices, recalling the construction of scalar-valued theta series and the classical Siegel--Weil formula for unimodular lattices. We then pass to non-unimodular lattices, where theta functions naturally become vector-valued modular forms for the Weil representation.

In the next step, we introduce the adelic formulation of the Siegel--Weil formula.
Although the adelic approach introduces some extra technical language,
this formulation is well-suited to simultaneously treat higher genus (Siegel) modular forms and quadratic spaces of arbitrary signature. The basic statement is that the theta integral obtained by averaging a theta function over an orthogonal group is equal to a special value of a specific Siegel Eisenstein series.

We then turn to the geometric version of the theory developed by Hirzebruch--Zagier, Oda, Kudla--Millson among others, see e.g.~\cite{hirzebruchzagier}, \textcolor{\newcol}{\cite{oda},} \cite{kudlamillson}. 
For quadratic spaces of signature $(n,2)$, the associated orthogonal Shimura varieties carry so-called special cycles.
These algebraic cycles include and generalize classical examples such as Heegner points, Hirzebruch--Zagier divisors on Hilbert modular surfaces, Humbert surfaces in Siegel threefolds, and Noether--Lefschetz loci in moduli spaces of K3-surfaces.
In the indefinite case, the representation numbers of any given lattice are typically infinite. The correct replacement is given by cohomology classes  of special cycles. The Kudla--Millson theory shows that the generating series of these classes converges to a modular form. The geometric Siegel--Weil formula identifies the generating series of their volumes with an Eisenstein series.

Finally, we discuss the arithmetic Siegel--Weil formula, which is part of Kudla's program on special cycles on integral models of Shimura varieties, see e.g.~\cite{Kudla-Annals}, \cite{Kudla-MSRI}. Here special cycles are replaced by arithmetic cycles in the sense of Gillet--Soulé, and one considers generating series taking values in arithmetic Chow groups. The analogue of the geometric volume is the arithmetic degree, and instead of an Eisenstein series one considers a derivative of a Siegel Eisenstein series (with respect to a spectral parameter) associated with an incoherent quadratic space.
 Thus the arithmetic Siegel--Weil formula predicts a relation between arithmetic intersection numbers and central derivatives of Siegel Eisenstein series.

Throughout these notes, we tried to keep the informal style of the lectures. In particular, we prioritize clarity over full generality and occasionally impose stronger hypotheses than necessary to simplify the exposition.
Furthermore, we assume that the reader is familiar with the theory of elliptic modular forms and the basics of Siegel modular forms.
Apart from the precourse notes of the summer school, 
we refer to~\cite{bump},  \cite{1-2-3}, \cite{koecherkrieg} for more details on the theory of elliptic modular forms, to~\cite{freitag},  \cite{klingen}  for background on Siegel modular forms, and to~\cite{pitale} for a more adelic and representation-theoretic point of view.

\subsection*{Acknowledgments}
\textcolor{\newcol}{We thank the referees for their careful reading of this paper and for their helpful comments.}
We are grateful to Claudia Alfes and Ana Botero for organizing the very stimulating summer school \emph{Formulas of Siegel and Weil} in Bielefeld.
We also thank Marco Eibrink for typesetting a first draft of these notes. Both authors are  supported in part by  the DFG Collaborative Research Centre TRR 326 ``Geometry and Arithmetic of Uniformized Structures'', project number 444845124.
The second author is supported by the DFG Research Grant (Eigene Stelle) ``Arithmetic and geometry of the Kudla--Millson theta function'', project number 554793187.

\section{Classical theta series and the Siegel--Weil formula}

Let $(L,Q)$ be an \emph{even lattice} over $\ZZ$, i.e., $L$ is a free $\ZZ$-module of rank $n$ and $Q\colon L\rightarrow \ZZ$ is a quadratic form.
By $(x,y)=Q(x+y)-Q(x)-Q(y)$ we denote the associated bilinear form. Moreover, we assume that $Q$ is non-degenerate, i.e., $L^\bot =\{0\}$.

A $\ZZ$-linear map of lattices $f\colon (L_1,Q_1)\rightarrow (L_2,Q_2)$ is called \emph{isometric} if $Q_2(f(x))=Q_1(x)$ for all $x\in L_1$.
Two lattices $(L_1,Q_1)$ and $(L_2,Q_2)$ are called \emph{isometric} if there exists an isometric isomorphism $(L_1,Q_1)\rightarrow (L_2,Q_2)$. In this case, we write $L_1\simeq L_2$.

For any prime $p$, we denote
\[
    L_p=L\otimes_\ZZ \ZZ_p\quad\text{and}\quad L_\infty=L\otimes_\ZZ\RR.
\]
The quadratic form~$Q$ of $L$ naturally extends to a quadratic form~$Q_p$ on $L_p$ for every prime~$p$ and $p=\infty$.

Let $b_1,\dots,b_n$ be a basis of $L$.
The \emph{Gram matrix} of $L$ with respect to that basis is the matrix $S=\left((b_i,b_j)\right)_{i,j}\in\ZZ^{n\times n}$.
The \emph{signature}~$\mathrm{sig}(L)$ and the \emph{determinant} $\mathrm{det}(L)$ of $L$ are respectively the signature and the determinant of~$S$.
Both values do not depend on the choice of the basis of~$L$, and are invariant under isometric isomorphisms.
We say that~$L$ is \emph{unimodular} if $\mathrm{det}(L)=\pm 1$.

\begin{Rem}\label{rem:unimodsign}
    An even unimodular lattice~$L$ of signature~$(b^+,b^-)$ exists if and only if~$b^+-b^-\equiv 0$ mod~$8$.
    In particular, every positive definite even unimodular lattice $L$ has rank~$n$ divisible by~$8$. See~\cite{serre} and~\cite{nikulin} for further details.
\end{Rem}

Note that~$\mathrm{sig}(L_1)=\mathrm{sig}(L_2)$ if and only if~$(L_1)_\infty \simeq(L_2)_\infty$.
More explicitly, if $\mathrm{sig}(L)=(b^+,b^-)$ with~$n=b^++b^-$, then~$L_\infty\simeq \RR^{b^+,b^-}$, where $\RR^{b^+,b^-}$ is the vector space $\RR^n$ endowed with the quadratic form
\[
Q(\vect{x})\coloneq x_1^2 + \dots +x_{b^+}^2 - x_{b^++1}^2 - \dots - x_n^2.
\]

\begin{Def}
    A lattice $(L_1,Q_1)$ is said to be \emph{in the genus of} $(L,Q)$ if $(L_1)_p\simeq L_p$ for all primes $p< \infty$ and for~$p=\infty$.
    We denote by $\operatorname{gen}(L,Q)$ the set of lattices in the genus of~$(L,Q)$.
\end{Def}

Let $L_\QQ=L\otimes_\ZZ \QQ$.

\begin{Rem}
    \begin{enumerate}
        \item If $(L_1,Q_1)\in\operatorname{gen}(L_2,Q_2)$, then $L_1\otimes_\ZZ \QQ_p\cong L_2\otimes_\ZZ \QQ_p$ for all $p\leq\infty$. By the Hasse--Minkowski Theorem~\cite[Theorem~6.2]{schulzepillot}, we deduce that $(L_1)_\QQ\simeq (L_2)_\QQ$. Therefore, we may assume that the lattices $L_1$ and $L_2$ are embedded in the same rational quadratic space.
        \item If two lattices are isometric, then they lie in the same genus. The converse is not true in general.
        \item
        Lattices in the same genus have the same determinant. Since by~\cite[Theorem~6.10]{schulzepillot} there are only finitely many isometry classes of lattices of fixed determinant and signature, the genus $\operatorname{gen}(L,Q)$ consists of finitely many isometry classes.
    \end{enumerate}
\end{Rem}
From now on we will frequently write $\operatorname{gen}(L)$ instead of $\operatorname{gen}(L,Q)$.

\begin{Def}
    The \emph{dual lattice} of $L$ is 
    \[
        L'\coloneq \left\{x\in L_\QQ\mid \; \text{$(x,y)\in\ZZ$\, for all $y\in L$}\right\}.
    \]
    Moreover, the \emph{level} of $L$ is defined as
    \[
        N\coloneq\min\left\{M\in\ZZ_{>0}\mid \; \text{$MQ(x)\in\ZZ$\, for all $x\in L'$}\right\}.
    \]
\end{Def}

One can show that:
\begin{itemize}
    \item 
The lattice $L$ is a finite index subgroup of $L'$.
\item $L$ is unimodular if and only if~$L=L'$.
Therefore, unimodular lattices are also called \emph{self-dual} lattices.
\item $L'/L$ is a finite abelian group, called the \emph{discriminant group}.
Its cardinality equals~$\lvert\mathrm{det}(L)\rvert$.
\item The quadratic form 
$Q$ induces a $\QQ/\ZZ$-valued quadratic form $\bar{Q}\colon L'/L\rightarrow \QQ/\ZZ$.
\item The discriminant group $L'/L$ decomposes as an orthogonal direct sum of $p$-subgroups as
        \[
            L'/L=\bigoplus_{p<\infty
            } L_p'/L_p.
        \]
\end{itemize}

The \emph{automorphism group} of the finite quadratic module $(L'/L,\bar{Q})$ is defined as 
\[
    \Orth(L'/L,\bar{Q})\coloneq\left\{f\in \mathrm{Aut}(L'/L)\,\mid\;  \bar{Q}(f(x))=\bar{Q}(x) \text{ for all } x\in L'/L\right\},
\]
where $\mathrm{Aut}(L'/L)$ denotes the group of $\ZZ$-linear automorphisms of~$L'/L$.
If the quadratic form is clear from the context, we will sometimes drop it from the notation and simply write~$\mathrm{O}(L'/L)$.

        For all $M\in\operatorname{gen}(L)$, we have $M'/M\simeq L'/L$ as finite quadratic modules.
        According to  \cite[Corollary 1.9.4]{nikulin} the converse also holds: If~$\operatorname{sig}(M)=\operatorname{sig}(L)$ and $M'/M\simeq L'/L$ as finite quadratic modules, then $M\in\operatorname{gen}(L)$.

\subsection{Theta series for positive definite lattices}
In this section we assume~$(L,Q)$ to be a \emph{positive definite} lattice of rank~$n$, so that~$\mathrm{sig}(L)=(n,0)$.
We also assume that~$n$ is \emph{even}.
Although one could drop the latter assumption and work with metaplectic double covers of congruence subgroups, for the sake of clarity we prefer to illustrate the theory in the simpler even rank case.

Let~$\HH=\{\tau\in\CC \,|\,\Im(\tau)>0\}$ denote the Poincaré upper half-plane.
\begin{Def}
    The \emph{theta series} associated with the lattice $L$ is defined by 
    \[
        \theta_L(\tau)=\sum_{x\in L}e^{2\pi i Q(x)\tau},\qquad \tau\in\HH. 
    \]
    More generally, for any coset $\mu\in L'/L$ we define
    \begin{equation}\label{eq:thetasermu}
        \theta_{\mu+L}(\tau)=\sum_{x\in\mu+L}e^{2\pi i Q(x)\tau }.
    \end{equation}
\end{Def}
\begin{Rem}
\begin{enumerate}
    \item Since $L$ is positive definite, the theta series $\theta_{\mu+L}$ converges normally on $\HH$.
    This implies that~$\theta_{\mu+L}$ is a holomorphic function with respect to~$\tau$, see~\cite[Section~2.4]{zuffetti} for details.
    \item Let $\Delta=(-1)^{n/2}|L'/L|=(-1)^{n/2}\mathrm{det}(L)$ be the \emph{discriminant} of~$L$.
    In this case we have $\Delta\equiv 0,1$ mod~$4$ and $\chi_\Delta (a)=\big(\frac{\Delta}{a}\big)$ is a quadratic Dirichlet character modulo $N$, where $N$ is the level of $L$.
    Here we denoted by~$\big(\frac{\,\cdot\,}{\cdot}\big)$ the Kronecker symbol.
\end{enumerate}
\end{Rem}

\begin{teor}[Theta transformation formula]\label{thm:thetatrfor}
   We have $\theta_L\in M_{n/2}(\Gamma_0(N),\chi_\Delta)$, i.e., $\Theta_L$ is a holomorphic modular form of weight $n/2$ for $\Gamma_0(N)$ with Nebentypus $\chi_\Delta$.
    Moreover, $\theta_{\mu+L}\in M_{n/2}(\Gamma(N))$, i.e.,  it is a holomorphic modular form for the principal congruence subgroup of level $N$.
\end{teor}
The dimension of $M_{n/2}(\Gamma_0(N),\chi_\Delta)$ is small when $n$ and $N$ are small.
If there are no non-trivial  cusp forms in this space, one can identify $\theta_L$ with an explicit linear combination of Eisenstein series.
In this case, one can deduce formulas for the \emph{representation numbers}
\[
    r_L(m)=\#\{x\in L\mid \; Q(x)=m\},
\]
for~$m\in\ZZ_{\geq0}$, using the well-known formulas for the Fourier coefficients of Eisenstein series.

We denote by $E_k$  the (normalized) \emph{Eisenstein series} for the full modular group $\Gamma(1)=\operatorname{SL}_2(\ZZ)$, defined as
\begin{align*}
    E_k(\tau)=\sum_{\left(\begin{smallmatrix}a&b\\ c&d\end{smallmatrix}\right)\in\Gamma_\infty\backslash\operatorname{SL}_2(\ZZ)}(c\tau+d)^{-k}=1-\frac{2k}{B_k}\sum_{n=1}^\infty\sigma_{k-1}(n)q^ n,\qquad q=e^{2\pi i\tau},\,\tau\in\HH.
\end{align*}
Here $B_k$ is the $k$-th Bernoulli number, $\sigma_{k-1}(n)=\sum_{d|n}d^{k-1}$ is the generalized divisor sum, and
\[
\Gamma_\infty =\left\{ \pm \begin{pmatrix} 1&m\\0 & 1\end{pmatrix} \,\mid \, m\in \ZZ \right\}
\]
is the subgroup of~$\SL_2(\ZZ)$ stabilizing the cusp $\infty$.

\begin{Exa}
    Let $L$ be the (unique up to isometry) even unimodular positive definite lattice of rank 8. Then Theorem~\ref{thm:thetatrfor} holds with $n=8$, $N=1$, $\Delta=1$, and $M_4(\Gamma(1))$ is the $1$-dimensional complex space spanned by the Eisenstein series $E_4$. Since the Fourier coefficients of index~$0$ of $E_4$ and $\theta_L$ coincide, we deduce that~$\theta_L=E_4$.
    Therefore, we obtain the explicit formula $r_L(m)=240\sigma_3(m)$ for the representation numbers of~$L$.
\end{Exa}

If there are non-trivial cusp forms, there are usually no such simple formulas for $r_L(m)$. However, one can still get asymptotic results for $r_L(m)$ as $m$ increases.
The purpose of the Siegel--Weil formula is to pinpoint a suitable average of the theta functions arising from the isometry classes of the lattices in $\mathrm{gen}(L)$, which turns out to be an Eisenstein series.

Let $\operatorname{O}(L)=\{f\colon L\rightarrow L\mid  f\textup{ isometry}\}$ be the orthogonal group of $L$.
Since~$L$ is positive definite, the group $\mathrm{O}(L)$ is finite.

\begin{teor}[Siegel]\label{thm:SWf-scalvalc}
    Let $L$ be an even positive definite unimodular lattice of rank $n$. Then
    \begin{equation}\label{eq;clsiegwe}
        \sum_{M\in\operatorname{gen}(L)/\simeq}\frac{1}{|\Orth(M)|}\theta_M(\tau)=\bigg(\sum_{M\in\operatorname{gen}(L)/\simeq}\frac{1}{|\Orth(M)|}\bigg)E_{n/2}(\tau).
    \end{equation}
\end{teor}
\begin{proof}
Recall from Remark~\ref{rem:unimodsign} that $n\equiv 0$ mod~$8$.
    Let $g_1$ and $g_2$ be respectively the left-hand side and the right-hand side of~\eqref{eq;clsiegwe}, and let $h\coloneq g_1-g_2$.
    We fix the weight~$k=n/2$.
    The theta series $\theta_M$ and the Eisenstein series $E_{k}$ are elements of $ M_{k}(\Gamma(1))$ with the same constant term in their Fourier expansion. Therefore, $h$ is a weight~$k$ cusp form for $\Gamma(1)$.
    
    Let $p$ be a prime and let $T_p$ be the $p$-th Hecke operator on $M_{k}(\Gamma(1))$, see~\cite[Part~I, Section~4.1]{1-2-3} for details. Using the explicit formula of the Fourier coefficients of $E_k$ and the way~$T_p$ acts on the Fourier coefficients of a modular form \cite[Part~I, (42)]{1-2-3}, we deduce that $T_p E_k=(p^{k-1}+1)E_k$.

    In order to prove the theorem we need the following two lemmas.
    \begin{lem}\label{lemma:inprclsiegwe}
       We have that $T_p(g_1)=(p^{k-1}+1)(g_1)$ for any prime $p>2$.
    \end{lem}
    Lemma~\ref{lemma:inprclsiegwe} follows from an explicit formula for the action of $T_p$ on the theta functions~$\theta_M$, see e.g.~\cite[Section~5.3.4]{koecherkrieg}. Note that the individual modular forms $\theta_M$ are not Hecke eigenforms in general!
    
    Lemma~\ref{lemma:inprclsiegwe} implies that
    \begin{equation}\label{eq:tphinpr}
        T_p(h)=(p^{k-1}+1)h.
    \end{equation}
    
    The next lemma can be proved in terms of elementary properties of the Petersson inner product, see e.g.~\cite{Ko} for a proof in much greater generality.
    \begin{lem}\label{lemma:inprclsiegwebis}
        If $f$ is a non-zero weight~$k$ cusp form for~$\Gamma(1)$ such that $T_p f=\lambda_p f$, then $|\lambda_p|<p^{k/2}+p^{k/2-1}$.
    \end{lem}
    Since $k\geq 4$, Lemma~\ref{lemma:inprclsiegwebis} and~\eqref{eq:tphinpr} imply that $h=0$.
\end{proof}

\section{A vector-valued generalization}\label{sec:vvSWclas}
In this section we explain in terms of vector-valued modular forms how theta series and the Siegel--Weil formula generalize to non-unimodular lattices.
For details, see~\cite[Section~9]{werner} and~\cite[Section~3.7]{opitz}.
We still work in the easier case of elliptic (i.e., genus~$1$) modular forms.
For a generalization to the case of (higher genus) Siegel modular forms see~\cite[Section~1.7]{manuel}.

Let $(L,Q)$ be an even lattice of rank $n$ and level $N$.
Similarly to the previous section, we assume for simplicity that $n$ is even.
Let~$\mathrm{sig}(L)=(b^+,b^-)$ denote the signature of~$L$, and let
\[
    \CC[L'/L]=\bigg\{\sum_{\mu\in L'/L}a_\mu e_\mu\mid a_\mu\in \CC\bigg\}
\]
be the group algebra generated by~$L'/L$.
The multiplication in~$\CC[L'/L]$ is defined by the law~$e_\mu\cdot e_\delta=e_{\mu + \delta}$ on the standard basis vectors of~$\CC[L'/L]$.

Endow~$\CC[L'/L]$ with the standard Hermitian inner product $\langle \sum_\mu a_\mu e_\mu,\sum_\delta b_\delta e_\delta\rangle\coloneq \sum_\mu a_\mu \overline{b_\mu}$.
The \emph{Weil representation} $\rho_L$ is the unitary representation of $\Gamma(1)$ on $\CC[L'/L]$ defined in terms of the standard generators of~$\Gamma(1)$ by
\begin{equation}\label{eq:weilreprho}
\begin{split}
    \rho_L\left(\left(\begin{matrix}1&1\\0&1\end{matrix}\right)\right)e_\mu&=e(Q(\mu))e_\mu\\
    \rho_L\left(\left(\begin{matrix}0&-1\\1&0\end{matrix}\right)\right)e_\mu&=\frac{e((b^--b^+)/8)}{\sqrt{|L'/L|}}\sum_{\delta\in L'/L} e(-(\mu,\delta))e_\delta
\end{split}
\end{equation}
for every $\mu\in L'/L$, where $e(z)\coloneq e^ {2\pi i z}$.
The representation~$\rho_L$ factors through $\Gamma(1)/\Gamma(N)\cong\operatorname{SL}_2(\ZZ/N\ZZ)$.

\begin{Def}
    A holomorphic function $f\colon \HH\rightarrow \CC[L'/L]$ is called a \emph{modular form for $\Gamma(1)$ of weight $k$ for~$\rho_L$} if:
    \begin{itemize}
        \item $f(\gamma\tau)=(c\tau+d)^ k\rho_L(\gamma)f(\tau)$  for all $\gamma=\left(\begin{matrix}a&b\\ c&d\end{matrix}\right)\in\Gamma(1)$.
        \item $f$ is holomorphic at the cusp $\infty$. 
    \end{itemize}
    We denote the space of such modular forms by $M_{k,\rho_L}$.
\end{Def}
We remark that if~$L$ is unimodular, then $\rho_L$ is trivial and $M_{k,\rho_L}$ boils down to the space of elliptic modular forms $M_k(\Gamma(1))$.

Every $f\in M_{k,\rho_L}$ has a Fourier expansion at $\infty$ of the form
        \[
            f(\tau)=\sum_{\substack{m\in\ZZ+Q(\mu)\\m\geq 0}}\sum_{\mu\in L'/L}c(m,\mu)q^ m e_\mu.
        \]
As in the scalar-valued case, the request that $f$ is holomorphic at~$\infty$ is equivalent to the fact that $c(m,\mu)$ vanishes whenever~$m<0$.

\begin{Def}
    The \emph{weight~$k$ Eisenstein series} with values in~$\CC[L'/L]$ (and arising from~$e_0$) is
\[
    E_{k,\rho_L}(\tau)=\sum_{\substack{\gamma\in\Gamma_\infty\backslash \Gamma(1) \\ \gamma=\big(\begin{smallmatrix}
        \ast & \ast \\ c & d
    \end{smallmatrix}\big)}}(c\tau+d)^{-k}\rho_L(\gamma)^{-1}e_0.
\]
\end{Def}
Similarly as in the scalar-valued setting, $E_{k,\rho_L}$ converges if $k>2$, defining a modular form in $M_{k,\rho_L}$.
It has a Fourier expansion of the form
    \[
        E_{k,\rho_L}(\tau)=e_0+\sum_{\mu\in L'/L}\sum_{\substack{m\in\ZZ+Q(\mu)\\m>0}}c(m,\mu)q^ m e_\mu
    \]
    with rational Fourier coefficients.
    The latter have been explicitly computed in~\cite{bruinierkuss}.
    We remark that it is possible to construct similar Eisenstein series from any~$e_\mu$ with $\mu\in L'/L$ isotropic.
    Since such Eisenstein series do not play any role in the Siegel--Weil formula, we avoid to introduce them here and instead refer to~\cite[Section~1.2.3]{bruinier-habil} and~\cite{bruinierkuss} for further details.

 The automorphism group $\mathrm{O}(L'/L)$ acts on $M_{k,\rho_L}$ by
    \[
        \sigma\cdot f\coloneq
        \sum_{\mu\in L'/L}f_\mu e_{\sigma\mu}, \qquad f=\sum_\mu f_\mu e_\mu\in M_{k,\rho_L},
    \]
    for~$\sigma\in \mathrm{O}(L'/L)$.
    Therefore, $M_{k,\rho_L}$ decomposes into isotypical subspaces for the irreducible representations of $\mathrm{O}(L'/L)$.
    Let $M_{k,\rho_L}^{\mathrm{sym}}$ be the subspace corresponding to the trivial representation.
    The modular forms in~$M_{k,\rho_L}^{\mathrm{sym}}$ are called \emph{symmetric}.
    The projection to the subspace of symmetric modular forms is given by
    \begin{align*}
        M_{k,\rho_L}&\rightarrow M_{k,\rho_L}^{\mathrm{sym}},\qquad 
        f\mapsto f^{\mathrm{sym}}\coloneq \frac{1}{|\mathrm{O}(L'/L)|}\sum_{\sigma\in \mathrm{O}(L'/L)}\sigma\cdot f.
    \end{align*}
    One can show that $E_{k,\rho_L}\in M_{k,\rho_L}^{\mathrm{sym}}$.

From now on, we assume that~$L$ is \emph{positive definite}, or equivalently that~$b^+=n$ and~$b^-=0$.
 We define the $\CC[L'/L]$-valued analogue of the theta series~\eqref{eq:thetasermu} as
    \[
        \theta_{L'/L}(\tau)=\sum_{\mu\in L'/L}\theta_{\mu+L}(\tau)e_\mu.
    \]
 Theorem~\ref{thm:thetatrfor} then generalizes as follows.
\begin{teor}
  We have that $\theta_{L'/L}\in M_{n/2,\rho_L}$.
\end{teor}

In contrast with $E_{k,\rho_L}$, the theta series $\theta_{L'/L}$ is in general not symmetric.
We are now ready to state the Siegel--Weil formula for these $\CC[L'/L]$-valued theta functions. 

\begin{teor}[Siegel--Weil formula]\label{thm_SWfvvp}
Let $L$ be an even positive definite lattice of even rank $n>4$.
    Then 
    \begin{equation}\label{eq;SWvvposdef}
        \sum_{M\in\operatorname{gen}(L)/\textcolor{\newcol}{\simeq}}\frac{1}{|\mathrm{O}(M)|}\theta_{M'/M}^{\mathrm{sym}}(\tau)=\bigg(\sum_{M\in\operatorname{gen}(L)/\textcolor{\newcol}{\simeq}}\frac{1}{|\mathrm{O}(M)|}\bigg)E_{n/2,\rho_L}.
    \end{equation}
\end{teor}
Note that for any~$M\in\operatorname{gen}(L)$, we have an isomorphism $M'/M\cong L'/L$ of finite quadratic modules.
In particular~$M_{k,\rho_M}\cong M_{k,\rho_L}$ for all weights~$k$.
Although there may be several choices of the isomorphism $M_{k,\rho_M}\cong M_{k,\rho_L}$, the induced isomorphism $M_{k,\rho_M}^{\mathrm{sym}}\cong M_{k,\rho_L}^{\mathrm{sym}}$ is independent of such choices.
This implies that we may consider the left-hand side of~\eqref{eq;SWvvposdef} as a combination of modular forms in $M_{n/2,\rho_L}^{\mathrm{sym}}$.

\begin{proof}[Proof of Theorem~\ref{thm_SWfvvp}]
    The proof is analogous to the one of Theorem~\ref{thm:SWf-scalvalc} in the scalar-valued case, but with one adaptation.
    Namely, we need  a Hecke theory on $M_{k,\rho_L}$. This has been developed in~\cite{bruinierstein}.
    To construct the Hecke operators $T_p$, one should extend $\rho_L$ to the monoid $\operatorname{Mat}_2^ +(\ZZ)$ of $2\times 2$ matrices with entries in $\ZZ$ and of positive determinant,
    or at least to some $\Gamma(1)$ double cosets in $\mathrm{Mat}^+_2(\ZZ)$.
    
    For the current proof, it suffices to consider primes $p\equiv 1$ mod~$N$, for which the construction of~$T_p$ is quite easy.
    In fact, in this case $\Gamma(1)\left(\begin{smallmatrix}p &0\\0&1\end{smallmatrix}\right)\Gamma(1)\equiv\Gamma(1)$ mod~$N$ and we may simply define
    \begin{align*}
        T_p(f)
        &=p^{k-1}\sum_{\gamma\in\Gamma(1)\backslash\Gamma(1)\left(\begin{smallmatrix}p&0\\0&1\end{smallmatrix}\right)\Gamma(1)} (c\tau+d)^{-k}\rho_L(\overline{\gamma})^{-1}f(\gamma\tau),
    \end{align*}
    where $\bar\gamma$ denotes the mod~$N$ reduction to~$\SL_2(\ZZ/N\ZZ)$ of~$\gamma$.
    Here we use the fact that~$\rho_L$ factors through $\operatorname{SL}_2(\ZZ/N\ZZ)$.
    
    By studying how~$T_p$ acts on the Fourier coefficients of modular forms for~$\rho_L$, one can check that $T_p(E_{k,\rho_L})=(p^{k-1}+1)E_{k,\rho_L}$ for such $p$.
    Lemmas~\ref{lemma:inprclsiegwe} and~\ref{lemma:inprclsiegwebis} hold also in this generality and one can conclude the proof as for Theorem~\ref{thm:SWf-scalvalc}.
\end{proof}

    The symmetrization operator $M_{k,\rho_L}\rightarrow M_{k,\rho_L}^{\mathrm{sym}}$ preserves the component $f_0$ of $f=\sum_\mu f_\mu e_\mu\in M_{k,\rho_L}$, in the sense that $(f^{\mathrm{sym}})_0=f_0$.
    Since~$(\theta_{L'/L})_0=\theta_L$, the equality of the $0$-th components of both sides of~\eqref{eq;SWvvposdef} and Theorem~\ref{thm:thetatrfor} imply the following result.
    \begin{Cor}
        Let $L$ be an even positive definite lattice of even rank $n>4$.
        Then
        \[
        \sum_{M\in\operatorname{gen}(L)/\textcolor{\newcol}{\simeq}}\frac{1}{|\mathrm{O}(M)|}\theta_M(\tau)=\bigg(\sum_{M\in\operatorname{gen}(L)/\textcolor{\newcol}{\simeq}}\frac{1}{|\mathrm{O}(M)|}\bigg)\left(E_{n/2,\rho_L}\right)_0\in M_{n/2}(\Gamma_0(N),\chi_\Delta).
    \]
    \end{Cor}
    \textcolor{\newcol}{
    Using the fact that~$\rho_L$ is trivial on~$\Gamma(N)$, one can show that $(E_{n/2,\rho_L})_0$ is a linear combination of scalar-valued Eisenstein series for~$\Gamma(N)$.
    This linear combination can be made explicit by means of Shintani's formula~\cite[Proposition~1.6]{shintani} for~$\langle \rho_L(M)e_0,e_0\rangle$.
    }
\section{Adelic version of the Siegel--Weil formula}\label{sec:adverSWF}
The results of the previous section can be generalized to higher genus Siegel theta functions, and to indefinite lattices.
At this level of generality, it is more convenient to describe the Siegel--Weil formula working over the adeles.
The main references are~\cite{weil},~\cite{kudlarallis} and~\cite{kudlarallisII}.

\subsection{The Weil representation}
Let $(V,Q)$ be a quadratic space over $\QQ$, with $Q$ a non-degenerate quadratic form.
In this section, we allow that~$Q$ may be indefinite of signature $(b^+,b^-)$.
For simplicity, we continue to assume that $\ell\coloneq\dim_\QQ V=b^++b^-$ is \emph{even}\textcolor{\newcol}{, see~\cite{sweet} for the case where~$\ell$ is odd}.
We denote by $H=\Orth(V)$ the group of isometries of~$(V,Q)$.

Let~$\A$ and~$\A_f$ be respectively the ring of adeles and of finite adeles over~$\QQ$, and let~$\hat{\ZZ}=\prod_p \ZZ_p\subset\A_f$.
We denote by~$\lvert\cdot\rvert_\A$ the adelic norm on~$\A$. 
We define the quadratic character~$\chi_V$ attached to $V$ as
\begin{align*}
    \chi_V\colon \A^\times/\QQ^\times&\rightarrow\{\pm1\},\qquad
    x\mapsto (x,(-1)^{\ell/2}\det(V))_\QQ,
\end{align*}
where~$(\cdot{,}\cdot)_\QQ$ denotes the global Hilbert symbol, see~\cite[Chapter~III]{serre} and~\cite{adeles} for further details.

The standard rational symplectic space of dimension $2r$ is the space $W=\QQ^{2r}$, where we consider any~$w\in W$ as a row vector~$w=(w_1,w_2)$ of column vectors $w_1,w_2\in\QQ^r$, endowed with the standard symplectic form $\left<(x_1,x_2),(y_1,y_2)\right>\coloneq {}^tx_1y_2-{}^tx_2y_1$. 
The \emph{symplectic group} $G=\operatorname{Sp}(W)$ \emph{of genus $r$} is the group of isomorphisms preserving $\langle\cdot{,}\cdot\rangle$.
Note that $G\cong\operatorname{Sp}_r$, where
\[
    \operatorname{Sp}_r(\QQ)=\left\{g\in\operatorname{GL}_{2r}(\QQ)\mid {}^tg J g=J\right\},\qquad J\coloneq\left(\begin{matrix}0&-1_r\\ 1_r&0\end{matrix}\right).
\]
We may consider~$\Sp_r$ as a generalization of~$\SL_2$, in the sense that $\Sp_1=\SL_2$.

The \emph{Siegel parabolic subgroup} $P=\left\{\left(\begin{smallmatrix}*&*\\ 0_r&*\end{smallmatrix}\right)\right\}\subset\operatorname{Sp}_r$ has a Levi decomposition $P=MN$, with Levi factor
\[
    M=\left\{m(a)=\left(\begin{matrix}a&0\\0&{}^ ta^{-1}\end{matrix}\right)\,\mid\, a\in\operatorname{GL}_r\right\}
\]
and unipotent radical subgroup
\[
    N=\left\{n(b)=\left(\begin{matrix}1&b\\0&1\end{matrix}\right)\,\mid\, b\in\operatorname{Sym}_r\right\}.
\]

To construct theta functions in this setting, we need to introduce (the Schrödinger model of) the Weil representation attached to the pair $(G,H)$.
The latter is a \emph{dual reductive pair}, in the sense that the groups~$G$ and~$H$ can be realized as subgroups of a larger symplectic group in which they are the centralizers of one another.
The standard way to construct such a larger symplectic group is to define it as the group of isometries of the space~$\WW=V\otimes_\QQ W$ endowed with the symplectic form 
\[
    \big\langle \big\langle v\otimes w,v'\otimes w'\big\rangle\big\rangle \coloneq (v,v')\cdot\langle w,w'\rangle.
\]

Let $\psi\colon \A/\QQ\rightarrow\CC^\times$ be the standard additive character.
This is constructed so that at the infinite place it is $\psi_\infty(x)=\textcolor{\newcol}{e(x)}$.
We now explain how to construct the Weil representation $\omega=\omega_\psi$ associated with~$(G,H)$ and $\psi$, referring to~\cite{weil} and~\cite{kudla-survey} for further details.
It is a representation of $G\times H$ on the space of Schwartz--Bruhat functions~$S(V(\A)^r)$ on~$V(\A)^r$.
Recall that
\begin{equation}\label{eq:splitschwsp}
    \begin{split}
         S(V(\A)^r)
    &=S(V(\A_f)^r)\otimes S(V(\RR)^r),
    \end{split}
\end{equation}
where~$S(V(\A_f)^r)$ is the space of locally constant functions of compact support on~$V(\A_f)^r$ and $S(V(\RR)^r)$ is the classical space of Schwartz functions on~$V(\RR)^r$.

The representation~$\omega$ is defined with respect to the standard generators of $G(\A)$ and~$H(\A)$ as
\begin{align*}
    (\omega(h)\varphi)(x)&=\varphi(h^{-1}x), \quad h\in H(\A),\\
    (\omega(m(a))\varphi)(x)&=\chi_V(\det(a))|\det(a)|^{\ell/2}_\A \varphi(xa),\quad a\in\mathrm{GL}_r(\mathbb{A}),\\
    (\omega(n(b))\varphi)(x)&=\psi(\operatorname{tr}(bQ(x)))\varphi(x), \quad b\in\operatorname{Sym}_r(\A),\\
    (\omega(S)\varphi)(x)&=
    \int_{V(\A)^r}\varphi(y)\psi(-(x,y))d_\psi y,
    \end{align*}
    wher $d_\psi y$ is the self-dual Haar measure on~$V(\A)^r$ with respect to $\psi$, and $S\coloneq\big(\begin{smallmatrix}0&-1_r\\1_r&0\end{smallmatrix}\big)\in\operatorname{Sp}_r(\A)$.

Under the decomposition~\eqref{eq:splitschwsp} of~$S(V(\A)^r)$, the Weil representation splits into local representations as~$\omega=\omega_f\otimes\omega_\infty$.
Here~$\omega_f$ and~$\omega_\infty$ are the Weil representations of~$G(\A_f)\times H(\A_f)$ and~$G(\RR)\times H(\RR)$ on respectively~$S(V(\A_f)^r)$ and~$S(V(\RR)^r)$.

\begin{Rem}\label{rem:weiltoweil}
    Let~$r=1$.
    The Weil representation~$\rho_L$ of~$\SL_2(\ZZ)$ on~$\CC[L'/L]$ defined in~\eqref{eq:weilreprho} can be considered as a restriction of~$\omega$, as follows.
    Denote by~$S_L$ the subspace of~$S(V(\A_f))$ spanned by the characteristic functions~$\varphi_{\mu,f}$ of the lattice cosets~$\mu+\hat L$, for~$\mu\in L'/L$, where~$\hat L=L\otimes_\ZZ\hat \ZZ$.
    This subspace is preserved under the action of~$\SL_2(\hat{\ZZ})$ given by~$\omega_f$. Furthermore, we have an isomorphism of group algebras
    \begin{equation}\label{eq:identgralg}
        \CC[L'/L]\to S_L,\qquad e_\mu\mapsto\varphi_{\mu,f}.
    \end{equation}
    Under the identification~\eqref{eq:identgralg}, for every~$\gamma\in\SL_2(\ZZ)$ one can show that
    \begin{equation}\label{eq_weilvsweil}
    \overline{\rho_L}(\gamma)e_\mu=\omega_f(\gamma)\varphi_{\mu,f}\qquad\text{for all $\mu$.}
    \end{equation}
    where~$\overline{\rho_L}$ is the conjugation of the representation~$\rho_L$.
    On the right-hand side of~\eqref{eq_weilvsweil} we identified~$\gamma$ with its image under the natural map~$\SL_2(\ZZ)\hookrightarrow\SL_2(\hat{\ZZ})$.
\end{Rem}

 We now construct a maximal compact subgroup $K'\coloneq K_\infty'\prod_{p< \infty}K_p'$ of $G(\A)$, where the local factors are constructed as follows.\footnote{We denote this subgroup by~$K'$ to avoid confusion, since $K$ will denote a compact subgroup of a general Spin group in Section~\ref{sec:geomSWF}.}
For any prime $p<\infty$, let $K_p'=\operatorname{Sp}_r(\ZZ_p)$.
For $p=\infty$, let
\begin{equation}\label{eq:compmaxofsympl}
    K_\infty'=\left\{k=\left(\begin{matrix}a&b\\-b&a\end{matrix}\right)\in \mathrm{Sp}_r(\RR)\,\mid\; \vect{k}\coloneq a+ib,\,\, \vect{k} {}^t\bar{\vect{k}}=1_r\right\}.
\end{equation}
The group~$K'$ acts on~$S(V(\A)^r)$ under the Weil representation.
We say that~$\varphi\in S(V(\A)^r)$ is \emph{$K'$-finite} if the space generated by all its $K'$-translates is finite dimensional, namely $\dim\langle\omega(k)\varphi \,|\,k\in K'\rangle_\CC<\infty$.


\subsection{Theta functions}
For every $\varphi\in S(V(\A)^r)$ we can construct a theta function on~$G(\A)\times H(\A)$ as
\[
    \theta(g,h,\varphi)=\sum_{x\in V(\QQ)^r}(\omega(g)\varphi)(h^ {-1}x),\qquad g\in G(\A),\, h\in H(\A).
\]
This theta function is left invariant under $H(\QQ)$ by construction and under $G(\QQ)$ by Poisson summation.
    Furthermore, it has moderate growth.
    
    We define the \emph{average value} of $\theta$ with respect to $H$ as
    \begin{equation}\label{eq:integralaveragegen}
        I(g,\varphi)=\int_{H(\QQ)\backslash H(\A)} \theta(g,h,\varphi)dh,
    \end{equation}
    where $dh$ is the invariant Haar measure on $H(\QQ)\backslash H(\A)$ normalized with $\operatorname{vol}(H(\QQ)\backslash H(\A))=1$.

We denote by~$\mathrm{Witt}(V)$ the \emph{Witt index} of~$V$. This is the dimension of any maximal isotropic subspace of~$V$. Note that~$\mathrm{Witt}(V)\leq\min(b^+,b^-)$, and that~$\mathrm{Witt}(V)=0$ whenever~$V$ is anisotropic.
\begin{prop}[Weil convergence criterion]\label{prop:weilconvcrit}
    The integral $I(g,\varphi)$ converges for all $\varphi$ if either
    \begin{enumerate}
        \item[i)] $V$ is anisotropic over $\QQ$.
        \item[ii)] $r<\ell-\mathrm{Witt}(V)-1$.
    \end{enumerate}
    If furthermore $\varphi$ is $K'$-finite, then $I(g,\varphi)$ is an automorphic form for $G$.
\end{prop}

The Siegel--Weil formula extends also to this setting: Under the Weil convergence criterion, the average value of $\theta$ is a special value of an adelic Eisenstein series. This is explained in the next section.

\subsection{Siegel Eisenstein series}\label{sec:SiegEisser}
Let $G(\A)=P(\A)K'=N(\A)M(\A)K'$ be the Iwasawa decomposition of $G(\A)$.
We write any $g\in G(\A)$ with respect to this decomposition as~${g=nm(a)k}$, and define 
\[
    \lvert a(g)\rvert\coloneq \lvert\mathrm{det}(a)\rvert_\A.
\]
Note that the function $g\mapsto |a(g)|$ on $G(\A)$ is left $N(\A)M(\QQ)$-invariant and right $K'$-invariant.

For $\varphi\in S(V(\A)^ r)$ and $s\in\CC$, we define
\begin{equation}\label{eq:standext}
    \phi(g,s)\coloneq\big(\omega(g)\varphi\big)(0)\cdot |a(g)|^{s-s_0},
\end{equation}
where
\[
s_0=\ell/2-\rho_r\qquad  \text{and}\qquad \rho_r=(r+1)/2.
\]
The defining formulas of the Weil representation for $N$ and $M$ imply that
\begin{equation}\label{eq:trainrep}
    \phi(nm(a)g,s)=\chi_V(\det(a))|\mathrm{det}(a)|_\A^{s+\rho_r}\phi(g,s).
\end{equation}
We denote by $I_r(s,\chi_V)=\operatorname{Ind}_{P(\A)}^{G(\A)}(\chi_V\cdot|\cdot|_\A^ s)$ the principal series induced representation of~$G(\A)$, consisting of all smooth functions on~$G(\A)$ satisfying the transformation~\eqref{eq:trainrep}.
Therefore, $\phi(s)\in I_r(s,\chi_V)$.

\begin{Rem}\label{rem:intertwining}
    The map $\lambda\colon S(V(\A)^r)\rightarrow I_r(s_0,\chi_V)$, $\varphi\mapsto \phi(s_0)$ is a \emph{$G(\A)$-intertwining map}, i.e., it is~$G(\A)$-equivariant.
    This is not true for other values of $s$, motivating why the value~$s_0$ is crucial in what follows.
    Here~$G(\A)$ acts on $S(V(\A)^r)$ and~$I_r(s_0,\chi_V)$ respectively under the Weil representation and by right-translations.
\end{Rem}
The (adelic) \emph{Siegel Eisenstein series} attached to $\phi(s)\in I_r(s,\chi_V)$ is defined as
\begin{align}
\label{eq:eis}
       E(g,s,\phi)=\sum_{\gamma\in P(\QQ)\backslash G(\QQ)}\phi(\gamma g,s),\qquad g\in G(\A).
\end{align}
It satisfies the following properties:
\begin{itemize}
    \item It converges absolutely for $\operatorname{Re}(s)>\rho_r$ if $\phi$ is \emph{standard}, i.e., its restriction to $K'$ is independent of $s$.
    \item \textcolor{\newcol}{For~$\varphi\in S(V(\A)^r)$, we denote by~$E(g,s,\lambda(\varphi))$ the Eisenstein series~\eqref{eq:eis} arising from the standard section~$\phi$ given by the extension of~$\lambda(\varphi)$ constructed in~\eqref{eq:standext}.}
    If~$\varphi$ is $K'$-finite, then $E(g,s,\lambda(\varphi))$ defines an automorphic form on $G(\A)$.
    \item By Langlands~\cite{langlands} it admits a meromorphic continuation in $s$ to the whole complex plane satisfying a functional equation under $s\mapsto -s$.
\end{itemize}

We are now ready to state the adelic version of the Siegel--Weil formula~\cite{kudlarallis}, \cite{kudlarallisII}.
\begin{teor}[Siegel, Weil, Kudla--Rallis]\label{thm;SiegelweilKR}
    Assume that the Weil convergence criterion holds. Let $\varphi\in S(V(\A)^r)$ be $K'$-finite
    .
    Then
    \begin{enumerate}
        \item[i)] $E(g,s,\lambda(\varphi))$ is holomorphic at $s=s_0$,
        \item[ii)] $E(g,s_0,\lambda(\varphi))=\kappa I(g,\varphi)$ where
        \[
            \kappa=\begin{cases}
                1,&\ell>r+1,\\
                2,&\ell\leq r +1.
            \end{cases}
        \]
    \end{enumerate}
\end{teor}

Roughly speaking, the idea of the proof is the following, see~\cite{marina} for details.
    \begin{itemize}
        \item Show that $I(g,\varphi)$ is orthogonal to cusp forms on $G(\A)$. Therefore, it must be a linear combination of Eisenstein series associated with parabolic subgroups of $G$.
        \item There is a lot of cancellation in the different contributions.
        \item Only the Siegel parabolic contribution survives.
    \end{itemize}

\subsection{Back to positive definite lattices}
\label{ex:explSWadcl}
    We briefly explain how Theorem~\ref{thm;SiegelweilKR} specializes to the case of positive definite lattices considered in Section~\ref{sec:vvSWclas}, for details see~\cite[Section~3.1.2]{opitz}.
    
    Let  $(V,Q)$ be positive definite of even dimension~$n>4$. For simplicity, we consider~$r=1$, so that $\operatorname{Sp}_r\cong\operatorname{SL}_2$.
    Assume that the Weil convergence criterion holds.
    Let $L\subset V$ be an even lattice.
    Fix a~$\mu\in L'/L$ and let $\varphi_{\mu,f}\in S(V(\A_f))$ be the characteristic function of the lattice coset~$\mu+\hat{L}$, as in Remark~\ref{rem:weiltoweil}.
    We use it to construct the Schwartz--Bruhat function~$\varphi_\mu=\varphi_{\mu,f}\otimes\varphi_\infty\in S(V(\A))$,
    where at the archimedean place we define~$\varphi_\infty\in S(V(\RR))$ to be the standard Gaussian $\varphi_\infty(x)=e^{-\pi(x,x)}$.
    The latter function is invariant under $H(\RR)$ and has weight $n/2$ under $K_\infty'=\operatorname{SO}_2(\RR)\subset\operatorname{SL}_2(\RR)$, in the sense that
    \begin{equation}\label{eq:weileigr1}
        \omega_\infty(\kappa_\theta)\varphi_\infty=e^{i\theta n/2}\varphi_\infty,\qquad\kappa_\theta=\begin{pmatrix}
        \cos\theta & \sin\theta \\ 
        -\sin\theta & \cos\theta
    \end{pmatrix}\in\mathrm{SO}_2(\RR).
    \end{equation}
    
    For $\tau\in\HH$, $\tau=u+iv$, we denote by~$g_\tau$ the standard element of~$\SL_2(\RR)$ mapping~$i\in\HH$ to~$\tau$, namely
    \begin{equation}\label{eq:gtauforgen1}
        g_\tau=n(u)m\big(v^{1/2}\big)=\begin{pmatrix}1&u\\0&1\end{pmatrix}\begin{pmatrix}v^{1/2}&0\\0&v^{-1/2}\end{pmatrix}.
    \end{equation}
    Recall from \cite[Section~1.1]{kiefer} and \cite[Section~2.5]{zuffetti} that to pass from a weight~$n/2$ automorphic form $f^\#\colon G(\A)\to\CC$ to a classical modular form~$f\colon\HH\to\CC$ of the same weight, one simply defines~$f(\tau)=j(g_\tau,i)^{n/2} f^\#(g_\tau)$.
    Note also that $H(\A_f)$ acts on cosets of lattices in $\operatorname{gen}(L)$ by $h (\mu + L)=V(\QQ)\cap h(\mu + \hat{L})$.
    In our current setting, we may pass from the adelic theta function attached to~$\varphi_\mu$ to a classical theta function over~$\HH$ as
    \begin{align*}
        j(g_\tau,i)^{n/2}\theta(g_\tau,h,\varphi_\mu)
        &=v^{-n/4}\sum_{x\in V(\QQ)}\varphi_{\mu,f}(h_f^{-1}x)(\omega(g_\tau)\varphi_\infty)(h_\infty^{-1}x)\\
        &=v^{-n/4}\sum_{x\in h_f(\mu+L)}(\omega(g_\tau)\varphi_\infty)(x)
        =\sum_{x\in h_f(\mu+L)}e^{2\pi i Q(x)\tau}=\theta_{h_f(\mu+L)}(\tau),
    \end{align*}
    where~$h=h_\infty h_f\in H(\RR)H(\A_f)$.


    We now consider the adelic Eisenstein series.
    The intertwining map and the Weil representation split respectively as~$\lambda=\lambda_f\otimes\lambda_\infty$ and~$\omega=\omega_f\otimes\omega_\infty$ under the product~$S(V(\A))=S(V(\A_f))\otimes S(V(\RR))$.
    Since~$P(\QQ)\backslash G(\QQ)\cong\Gamma_\infty\backslash\SL_2(\ZZ)$, we may rewrite the Eisenstein series as
    \begin{align*}
        j(g_\tau,i)^{n/2} E(g_\tau,s_0,\lambda(\varphi_\mu))&=
        v^{-n/4}\sum_{\gamma\in \Gamma_\infty\backslash\operatorname{SL}_2(\ZZ)}\lambda(\varphi_\mu)(\gamma g_\tau,s_0)\\
        &=v^{-n/4}\sum_{\gamma\in \Gamma_\infty\backslash\operatorname{SL}_2(\ZZ)}\lambda_f(\varphi_{\mu,f})(\gamma,s_0)
        \cdot \lambda(\varphi_\infty)(\gamma g_\tau,s_0)
        \\
        &=
        v^{-n/4}\sum_{\gamma\in \Gamma_\infty\backslash\operatorname{SL}_2(\ZZ)}
        (\omega_f(\gamma)\varphi_{\mu,f})(0)
        \cdot
        (\omega_\infty(\gamma g_\tau)\varphi_\infty)(0)
        .
    \end{align*}
    We compute the last two factors separately, beginning with the archimedean one.
    Since $\mathrm{SO}_2(\RR)$ is the stabilizer in~$\SL_2(\RR)$ of the point~$i$, there exists a~$\kappa_\theta\in\mathrm{SO}_2(\RR)$ such that~$\gamma g_\tau=g_{\gamma\tau}\kappa_\theta$.
    Since
    \begin{align*}
    (c\tau+d)v^{-1/2}
    &=j(\gamma,\tau) j(g_\tau,i)=j(\gamma g_\tau,i)=j(g_{\gamma\tau}\kappa_\theta,i)=j(g_{\gamma\tau},i)j(k_\theta,i)
    \\
    &=
    \lvert c\tau+d\rvert v^{-1/2} j(k_\theta,i),
    \end{align*}
    we deduce that $e^{-i\theta}=j(\kappa_\theta,i)=(c\tau+d)/\lvert c\tau+d\rvert$.
    By~\eqref{eq:weileigr1} we then compute
    \begin{align*}
         (\omega_\infty(\gamma g_\tau)\varphi_\infty)(0)
         &=
         (\omega_\infty(g_{\gamma\tau})\omega_\infty(\kappa_\theta)\varphi_\infty)(0)
         =\Big(\frac{c\tau+d}{\lvert c\tau+d\rvert}\Big)^{-n/2}(\omega_\infty(g_{\gamma\tau})\varphi_\infty)(0)
         \\
         &=v^{n/4}\Big(\frac{c\tau+d}{\lvert c\tau+d\rvert}\Big)^{-n/2} \lvert c\tau+d\rvert^{-n/2}\varphi_\infty(0)
         =v^{n/4}(c\tau+d)^{-n/2}.         
    \end{align*}

    We now consider the non-archimedean factor.
    By Remark~\ref{rem:weiltoweil}, the value~$(\omega_f(\gamma)\varphi_{\mu,f})(0)$ is the coefficient of~$e_\mu$ in the rewriting of~$\rho_L(\gamma)^{-1} e_0$ over the standard basis of~$\CC[L'/L]$.
    Summarizing, we deduce that
    \[
    j(g_\tau,i)^{n/2} E(g_\tau,s_0,\lambda(\varphi_\mu))= \mu\textup{-th component of }E_{n/2,\rho_L}.
    \]

    
    Since $V$ is positive definite, the theta integral over $H(\QQ)\backslash H(\A)$ reduces to the sum over lattices $M\in\operatorname{gen}(L)$, weighted by $\Orth(M)$.
    More precisely, decompose
    \[
        \Orth(V)(\A_f)=\bigcup_{j=1}^{h(L)}\Orth(V)(\QQ)h_j \Orth(\hat{L})
    \]
    for suitable $h_j\in \Orth(V)(\A_f)$, where~$h(L)$ denotes the \emph{class number} of $L$, i.e.,  the number of isometry classes in~$\mathrm{gen}(L)$. 
    Note that $\Orth(h_jL)=\Orth(V)\cap(\Orth(V)(\RR)h_j \Orth(V)(\A_f)h_j^{-1})$. Then
    \begin{align*}
        &v^{-n/4}\int_{H(\QQ)\backslash H(\A)}\theta(g_\tau,h,\varphi_\mu)dh
        \\
        &\quad=\mu\textup{-th component of }\bigg(\sum_j\frac{1}{|\Orth(h_jL)|}\bigg)^{-1}\sum_{j=1}^{h(L)}\frac{1}{|\Orth(h_jL)|}\theta_{h_j L'/h_jL}^{\mathrm{sym}}(\tau).
    \end{align*}
    This recovers Theorem~\ref{thm_SWfvvp}, namely the Siegel--Weil formula for positive definite lattices.

\section{The geometric Siegel--Weil formula}\label{sec:geomSWF}

In the previous sections, we saw that, for a positive definite lattice $L$, the representation numbers $r_L(m)$ are finite and appear as Fourier coefficients of certain theta series.
The modularity of the latter gives a way to study the arithmetic of quadratic forms in terms of properties of modular forms.
The Siegel--Weil formula shows that the average value of the theta functions arising from the isometry classes in the genus of~$L$ is an Eisenstein series.

In the indefinite case, however, the representation numbers are typically infinite.
Geometrically, this reflects the non-compactness of the real orthogonal group $\Orth(L)(\mathbb{R})$.
Instead, one can replace the representation \emph{numbers} with the \emph{cohomology classes} of certain algebraic cycles, known as ``special cycles'', on orthogonal Shimura varieties.
As proved by Kudla and Millson in much greater generality \cite{kudlamillson}, the generating series of such cycles behave as modular forms with values in cohomology groups.

The starting point for these developments was the celebrated work  of Hirzebruch and Zagier~\cite{hirzebruchzagier}, who showed that the intersection numbers of Hirzebruch--Zagier divisors on Hilbert modular surfaces occur as Fourier coefficients of certain elliptic holomorphic modular forms of weight~$2$.
Kudla and Millson~\cite{kudlamillson} vastly generalized this phenomenon to Shimura varieties (and more generally locally symmetric spaces) of orthogonal type and to special cycles of higher codimension, using theta functions to construct cohomology-valued modular generating series; see also \cite[Section~7]{kudla-algcy} for a summary.

In this framework, the \emph{geometric Siegel--Weil formula} relates the generating series of \emph{volumes} of special cycles to Eisenstein series.

\subsection{Orthogonal Shimura varieties}
\label{sect:5.1}
Let $(V,Q)$ be a quadratic space over $\QQ$ of signature~$(n,2)$. For simplicity, assume that $n$ is \emph{even}.
Let $L\subset V$ be an even lattice.

It is sometimes more convenient to replace~$\operatorname{SO}(V)$ with the general Spin group 
\[
    H=\operatorname{GSpin}(V)=\{g\in C^0(V)^\times\mid gVg^{-1}=V\},
\]
which is a reductive algebraic group defined over $\QQ$. Here~$C^0(V)$ denotes the even Clifford algebra, see e.g.~\cite[Part II, Section~2]{1-2-3} for details.
We may regard $H$ as a central extension of~$\operatorname{SO}(V)$, fitting into a short exact sequence
\[
    1\rightarrow Z\coloneq \mathbb{G}_m\rightarrow \operatorname{GSpin}(V)\rightarrow\operatorname{SO}(V)\rightarrow 1.
\]

By
\[
    D=\{z\in V_\CC\mid \; (z,z)=0, (z,\bar{z})<0\}/\CC^\times=D^+\cup D^-
\]
we denote the corresponding Hermitian symmetric space, on which $\operatorname{SO}(V\otimes\RR)$ (and hence~$H(\RR$)) acts transitively.
It has two connected components, denoted~$D^+$ and~$D^-$.
Let~$K\subset H(\A_f)$ be a compact open subgroup.
A standard choice is
\begin{equation}\label{eq:stdchK}
    K=C^0(\hat{L})^\times\cap H(\A_f),
\end{equation}
whose image under the map~$H(\A_f)\to\mathrm{SO}(V)(\A_f)$ is the \emph{discriminant kernel} of~$L$, i.e., the largest compact open subgroup that stabilizes~$\hat{L}$ and acts as the identity on $\hat{L}'/\hat L$.


The double quotient
\[
X_K(\CC)\coloneq H(\QQ)\backslash D\times H(\A_f)/K
\]
is the complex space of a \emph{Shimura variety} $X_K$ defined over $\QQ$.
By the theorem of Baily and Borel, it is a quasi-projective variety of $\dim X_K=n$.

Let $H(\RR)^+$ be the connected component of $H(\RR)$ containing the identity.
There are finitely many $h_j\in H(\A_f)$
such that $H(\A_f)$ decomposes into a disjoint union of the form
\[
    H(\A_f)=\bigcup_j H(\QQ)H(\RR)^+h_jK.
\]
Then we may split $X_K(\CC)$ in connected components as
\[
    X_K(\CC)=\bigcup_j X_j,
\]
where
\begin{equation}\label{eq:concompshim}
    X_j=\Gamma_j\backslash D^+ \qquad\text{and}\qquad\Gamma_j=H(\QQ)\cap(H(\RR)^+h_jKh_j^{-1}).
\end{equation}

Till the end of Section~\ref{sec:geomSWF}, we will mostly work with the space~$X_K(\CC)$ of complex points of~$X_K$. Therefore, to simplify the notation, we usually drop~$\CC$ and simply write~$X_K$.

We remark that:
\begin{itemize}
    \item If $K$ is chosen as in~\eqref{eq:stdchK} and~$L$ splits a hyperbolic plane, then $X_K$ is connected.
    \item We say that~$K$ is \emph{neat} if the multiplicative subgroup of~$\CC^\times$ generated by the eigenvalues of the elements of~$K$ has no torsion.
    This prevents every~$\Gamma_j$ in~\eqref{eq:concompshim} from having torsion, and makes~$X_K$ a \emph{smooth} variety.
    The open compact~\eqref{eq:stdchK} is not neat in general.
\end{itemize}

\begin{Exa}\label{ex:Hilbertmodsurf}
    Let $\QQ\subset F$ be a real quadratic extension, and let $\mathcal{O}_F$ be its ring of integers. Let
    \[
    V=\{x\in F^{2\times 2}\mid {}^tx=x'\},
\]
where $x'$ denotes the conjugate of $x$, and define the quadratic form $Q(x)=\det(x)$.
We consider $V$ as a quadratic space over~$\QQ$. It is easy to see that $\dim_\QQ (V)=4$ and $V$ has signature $(2,2)$.
In this case, we have that
\[
    H(\QQ)\cong\{g\in \operatorname{GL}_2(F)\mid\; \det(g)\in\QQ^\times\}\qquad\text{and}\qquad D^+\cong\HH\times\HH.
\]
For $K$ as in~\eqref{eq:stdchK}, the connected component of the Shimura variety~$X_K$ arising from~$h_j=1\in H(\A_f)$ is the \emph{Hilbert modular surface}~$\operatorname{SL}_2(\mathcal{O}_F)\backslash\HH^2$.
\end{Exa}


\subsection{Special cycles}\label{sec:specialcycles}
Let $x\in V$ with $Q(x)>0$. Then $V_x=x^\bot$ is a quadratic space of signature $(n-1,2)$. Let $H_x$ be the stabilizer of $x$ in $H$ and $D_x=\{z\in D\mid z\ \bot \ x\}$.
For~$h\in H(\A_f)$ we put $K_{h,x}=H_x(\A_f)\cap hKh^{-1}$. The map
\begin{equation}\label{eq:projmapncyc}
\begin{split}
    H_x(\QQ)\backslash D_x\times H_x(\A_f)/K_{h,x}&\rightarrow X_K\\
    [z,h_1]&\mapsto[z,h_1h]
\end{split}
    \end{equation}
defines a codimension 1 algebraic cycle $Z(x,h)$ in the Shimura variety $X_K$, which is defined over~$\QQ$.
Note that the left-hand side of~\eqref{eq:projmapncyc} is (the set of complex points of) an orthogonal Shimura variety of dimension~$n-1$.
Therefore, $Z(x,h)$ may be viewed as a Shimura subvariety of the same type as $X_K$.

The construction above easily generalizes to higher codimension as follows.
Let~$r\leq n$.
For any tuple $x=(x_1,\ldots ,x_r)\in V^r$ with positive definite moment matrix $Q(x)=\frac{1}{2}((x_i,x_j))_{i,j}\in\QQ^{r\times r}$, we define a quadratic space $V_x=\cap_jx_j^\perp$ of signature~$(n-r,2)$ and denote by~$H_x$ the pointwise stabilizer  in~$H$ of the span of the entries of~$x$.
We define~$K_{h,x}=H_x(\A_f)\cap hKh^{-1}$, so that now the projection~\eqref{eq:projmapncyc} gives rise to an algebraic cycle~$Z(x,h)$ of codimension~$r$ in $X_K$.
For more details on these cycles and the ones constructed in the upcoming section, we refer to ~\cite{kudla-algcy}.

\subsection{Weighted special cycles}
We now define some cycles that arise as weighted formal sums of the cycles constructed in Section~\ref{sec:specialcycles} under the assumption that~$K$ is \emph{neat}.
The reason for this assumption is that it makes the definition of the cycles in \eqref{eq:weightedspcy} below easier: For general~$K$ one may have to correct the multiplicities of the irreducible components of such cycles to obtain a generating series with good modularity properties.

Let $T\in\operatorname{Sym}_r(\QQ)$ be positive definite, and let $\varphi\in S(V(\A_f)^r)^K$, i.e., $\varphi$ is $K$-invariant. If there exists $x\in V^r$ such that $Q(x)=T$, we construct the codimension~$r$ \emph{weighted special cycle}~$Z(T,\varphi)$ to be the finite formal sum with complex coefficients
\begin{equation}\label{eq:weightedspcy}
    Z(T,\varphi)\coloneq\sum_{h\in H_x(\A_f)\backslash H(\A_f)/K}\varphi(h^{-1}x)Z(x,h)
    .
\end{equation}
If there is no such $x$, then we define $Z(T,\varphi)=0$.
By Witt's theorem, the definition of~$Z(T,\varphi)$ does not depend on the choice of~$x\in V^r$ with~$Q(x)=T$.

By Poincaré duality, see e.g.~\cite[Remark~2.6.3]{zuffetti}, the cycle~$Z(T,\varphi)$ induces a cohomology class $[Z(T,\varphi)]$ in the de Rham cohomology group $H^{2r}(X_K,\CC)$.

We now extend the previous definition to matrices $T\geq 0$, i.e., matrices that are positive semidefinite but possibly singular.
The problem of~\eqref{eq:weightedspcy} is that the cycle is now of codimension~$\mathrm{rk}(T)$ rather than~$r$.
To construct a cohomology class of degree~$2r$, we proceed as follows.
Let~$\mathcal{L}$ and $\mathcal{L}^\vee$ be respectively the tautological bundle of $X_K$ and its dual bundle. They induce cohomology classes $[\mathcal{L}]$ and~$[\mathcal{L}^\vee]\in H^2(X_K,\CC)$ with the usual property that~$[\mathcal{L}^\vee]=-[\mathcal{L}]$. The line bundle~$\mathcal{L}$ is also known as the \emph{Hodge bundle}.
Then we define
    \[
        [Z(T,\varphi)]\coloneq \Big(\sum_{h\in H_x(\A_f)\backslash H(\A_f)/K}\varphi(h^{-1}x)[Z(x,h)]\Big)\cup [\mathcal{L}^\vee]^{r-\mathrm{rk}(T)} \in H^{2r}(X_K,\CC).
    \]

\begin{Rem}
\textcolor{\newcol}{
These weighted cycles were constructed by Kudla in~\cite{kudla-algcy}.
They have the nice property that, if~$K_0\subset K$ is another compact open subgroup of~$H(\A_f)$, then the weighted cycle of index~$(T,\varphi)$ on~$X_{K_0}$ is the pullback under the covering map~$X_{K_0}\to X_K$ of the cycle on~$X_K$ with the same index.
This explains why we do not keep track of~$K$ in the notation~$Z(T,\varphi)$.
}
 \end{Rem}

\subsection{Generating series of cycles}
Let~$K$ be neat and let~$\varphi\in S(V(\A_f)^r)$ be~$K$-invariant, as above.
We denote by~$\HH_r$ the Siegel upper half-space in genus~$r$, namely the set of symmetric complex~$r\times r$ matrices with positive definite imaginary part.
It is the Hermitian symmetric domain attached to~$\Sp_r$.

The \emph{generating series of weighted special cycles} in codimension~$r$ is the formal series
    \[
        \Theta_r^{\mathrm{coh}}(\tau,\varphi)=\sum_{T\geq 0}[Z(T,\varphi)]q^T,\quad \text{where}\quad  q^{\textcolor{\newcol}{T}}\coloneq e^{2\pi i \operatorname{tr}(T\tau)} \quad\text{and}\quad\tau\in\HH_r.
    \]
    This looks like a Fourier expansion of a genus~$r$ Siegel modular form, but a priori it could be not even convergent.
    
The following result of Kudla and Millson~\cite{kudlamillson} ensures that~$\Theta_r^{\mathrm{coh}}$ is the $q$-expansion of a Siegel modular form with values in $H^{2r}(X_K,\CC)$, in the sense that
\[
\Lambda(\Theta_r^{\mathrm{coh}})=\sum_{T\geq 0}\Lambda\big([Z(T,\varphi)]\big)q^T
\]
is a scalar-valued modular form for every linear functional~$\Lambda\colon H^{2r}(X_K,\CC)\to\CC$.
In particular, if~$\eta$ is a degree~$2(n-r)$ compactly supported closed differential form on~$X_K$, then~$\int_{X_K}\Theta_r^{\mathrm{coh}}\wedge \eta$ is a scalar-valued modular form.

\begin{teor}[Kudla--Millson]\label{KM}
Let $\varphi\in S(V(\A_f)^r)^K$.
  The series  $\Theta_r^{\mathrm{coh}}(\tau,\varphi)$ is a holomorphic Siegel modular form of weight $1+n/2$ and genus $r$.
  More precisely, let~$\omega_f$ be the Weil representation of $\operatorname{Sp}_r(\A_f)$ on~$S(V(\A_f)^r)$, and consider the modular group $\operatorname{Sp}_r(\ZZ)$ embedded diagonally in~$\operatorname{Sp}_r(\A_f)$. Then 
    \[
        \Theta_r^{\mathrm{coh}}(\gamma\tau,\varphi)=\det(c\tau+d)^{1+n/2}\Theta_r^{\mathrm{coh}}(\tau,\omega_f(\gamma^{-1})(\varphi))
    \]
    for all $\gamma \in \mathrm{Sp}_r(\ZZ)$.
\end{teor}
For a sufficiently large~$N$, $\varphi$ is~$\Gamma_r(N)$-invariant under~$\omega_f$, where~$\Gamma_r(N)=\ker(\Sp_r(\hat{\ZZ})\to\Sp_r(\ZZ/N\ZZ))$.
In particular, Theorem~\ref{KM} implies that~$\Theta_r^{\mathrm{coh}}(\varphi)$ is a Siegel modular form with level.

\begin{proof}[Strategy of the proof]
The idea is to realize $\Theta_r^{\mathrm{coh}}$ as the cohomology class of a theta function with values in closed differential forms.
Such a theta function arises from a particular Schwartz function~$\varphi_{\mathrm{KM}}^r$, nowadays known as the \emph{Kudla--Millson Schwartz form}, constructed in~\cite{kudlamillsonI} and~\cite{kudlamillsonII}.
    It is a Schwartz function on~$x\in V(\RR)^r$ with values in differential~$2r$-forms on~$D$, more precisely
    \[
        \varphi_{\mathrm{KM}}^r\in[S(V(\RR)^r)\otimes A^{2r}(D)]^{H(\RR)},
    \]
    where~$A^{2r}(D)$ denotes the space of differential~$2r$-forms on~$D$ and the superscript~$H(\RR)$ means that $\varphi_{\mathrm{KM}}^r$ is $H(\RR)$-invariant, in the sense that $h^*\varphi_{\mathrm{KM}}^r(hx)=\varphi_{\mathrm{KM}}^r(x)$ for all~$x\in V(\RR)^r$ and~$h\in H(\RR)$.
    Such a Schwartz form is \emph{closed} with respect to the exterior differential on~$D$, and for fixed~$x$ the form~$\varphi_{\mathrm{KM}}^r (x)\in A^{2r}(D)$ is Poincaré dual to the cycle $D_x\subset D$.
    Furthermore, it has weight $k=1+n/2$ with respect to the standard maximal compact subgroup $U(r)\subset\operatorname{Sp}_r(\RR)$.

Therefore, the theta series arising from~$\varphi_{\mathrm{KM}}^r$ constructed as usual as
    \begin{equation}\label{eq:inproofthetaKM}
        \theta_{\mathrm{KM},r}(\tau,z,h,\varphi)=\det(v)^{-k/2}\sum_{x\in V(\QQ)^r}\varphi(h^{-1}x)(\omega_\infty(g_\tau)\varphi_{\mathrm{KM}}^r)(x,z),\qquad \tau=u+iv,
    \end{equation}
    is a \emph{non-holomorphic} modular form of weight $k$ with representation $\omega_f^\vee$ for $\operatorname{Sp}_r(\ZZ)$.
    Here~$g_\tau=n(u)m(a)$ is the generalization of~\eqref{eq:gtauforgen1} to higher genus, where~$a\in\mathrm{GL}_r(\RR)^+$ is chosen such that $a\, {^ta}=v$, so that~$g_\tau (i1_r)=\tau$.
    We may think of~\eqref{eq:inproofthetaKM} as a differential form on~$X_K$ with respect to the variables~$(z,h)$.
    
    In~\cite{kudlamillson} it is proved that, after passing to cohomology, the class~$[\theta_{\mathrm{KM},r}(\tau,h,\varphi)]$ becomes \emph{holomorphic} with respect to~$\tau$, and equals the generating series~$\Theta_r^{\mathrm{coh}}(\tau,\varphi)$.
\end{proof}
These generating series behave well with respect to the cup product in cohomology, as illustrated in the following result, see~\cite[Theorem~6.2]{kudla-algcy}.

\begin{teor}[Kudla--Millson]
    Let $r_1,r_2\in\ZZ_{\geq 0}$ be such that $r_1+r_2=r$. Let $\varphi_j\in S(V(\A_f)^{r_j})^K$. Then 
    \[
        \Theta_{r_1}^{\mathrm{coh}}(\tau_1,\varphi_1)\cup\Theta_{r_2}^{\mathrm{coh}}(\tau_2,\varphi_2)=\Theta_r^{\mathrm{coh}}\left(\left(\begin{matrix}\tau_1&0\\ 0&\tau_2\end{matrix}\right),\varphi_1\otimes\varphi_2\right).
    \]
    In particular,
    \[
        [Z(T_1,\varphi_1)]\cup[Z(T_2,\varphi_2)]=\sum_{\substack{T\in\operatorname{Sym}_r(\QQ)\\
        T=\left(\begin{smallmatrix}T_1&*\\ *&T_2\end{smallmatrix}\right)\geq0}}[Z(T,\varphi_1\otimes\varphi_2)].
    \]
\end{teor}
\begin{proof}[Idea of the proof]
    It follows from a similar formula for the theta functions~\eqref{eq:inproofthetaKM} arising from the Kudla--Millson Schwartz forms, induced by the identity~$\varphi_{\mathrm{KM}}^{r_1}(x_1,z)\wedge\varphi_{\mathrm{KM}}^{r_2}(x_2,z)=\varphi_{\mathrm{KM}}^r((x_1,x_2),z)$, see~\cite{kudlamillsonI} and~\cite{kudlamillsonII}.
\end{proof}

\subsection{The geometric Siegel--Weil formula}
We explain here how the Siegel--Weil formula generalizes to the case of generating series of special cycles.
The idea is to replace the average integral \eqref{eq:integralaveragegen} of the theta function on $X_K$ with an integral of the generating series~$\Theta_r^{\mathrm{coh}}(\varphi)$ of special cycles cupped with a suitable power of a Kähler form on~$X_K$.

We denote by $\Omega$ the first Chern form of the line bundle $\mathcal{L}$, where the latter is endowed with a certain Hermitian metric arising from the bilinear form of~$V$ as in~\cite[(1.5)]{kudla-integrals}.
The form~$\Omega$ is a $H(\RR)$-invariant \emph{Kähler form}. The induced Kähler metric on~$X_K$ is such that the \emph{volume} of the codimension~$r$ special cycles
\[
\mathrm{deg}(Z(T,\varphi))
=
\int_{Z(T,\varphi)}\Omega^{n-r}
=
(-1)^{r-\mathrm{rk}(T)}\sum_{h\in H_x(\A_f)\backslash H(\A_f)/K}\varphi(h^{-1}x)\int_{Z(x,h)}\Omega^{n-\mathrm{rk}(T)},
\]
for $T$ positive semidefinite of rank~$\mathrm{rk}(T)$, is finite.
Here~$x$ is any $x\in V^r$ such that $Q(x)=T$, as in~\eqref{eq:weightedspcy}.

Recall that a~$L^2$-differential form on~$X_K$ is a smooth square-integrable form whose exterior derivative is still square-integrable.
Here the square-integrability is with respect to the Kähler structure induced by~$\Omega$.
The Kähler form~$\Omega$ is an example of an $L^2$-form.

We denote by~$\sqH^*(X_K,\CC)$ the cohomology of the complex of $L^2$-forms.
The inclusion of such complex in the de Rham complex induces homomorphisms
\begin{equation}\label{eq:fromL2intoderham}
\sqH^{2r}(X_K,\CC)\to H^{2r}(X_K,\CC),
\end{equation}
which in general are neither injective nor surjective. We denote the $L^2$-, resp.\ de Rham, cohomology class of a~$L^2$-form~$\alpha$ by~$[\alpha]_{(2)}$, resp.~$[\alpha]$.

We define the \emph{degree map} as
\begin{align*}
    \operatorname{deg}\colon \sqH^{2r}(X_K,\CC)&\rightarrow \CC,
    \qquad [\alpha]_{(2)}\mapsto \int_{X_K}\alpha\wedge\Omega^{n-r}.
\end{align*}
Since~$X_K$ is a complete Kähler manifold, Stokes' theorem holds by \cite{gaffney}. Hence the integral on~$X_K$ of a top-degree exact $L^2$-form vanishes and the integral $\deg([\alpha]_{(2)})$ does not depend on the choice of the representative of the cohomology class~$[\alpha]_{(2)}$.

For simplicity, from now on we assume that~$r<\frac{n-1}{2}$.
Under this condition, the homomorphism~\eqref{eq:fromL2intoderham} is \emph{bijective}, see~\cite[Section~3.1]{BLMM} for details in greater generality.
Recall that~$\Theta_r^{\mathrm{coh}}(\varphi)$ is realized as the cohomology class of the Kudla--Millson theta form~$\theta_{\mathrm{KM},r}(\varphi)$ defined in~\eqref{eq:inproofthetaKM}, which is non-holomorphic in the variable~$\tau$.
This theta form is~$L^2$, see e.g.~\cite[Section~4]{bruinierzuffetti}.
Let
\[
\theta_{\mathrm{KM},r}(\tau,\varphi)=\sum_{T} c_T(v,\varphi) e^{2\pi i \tr Tu},\qquad \tau=u+iv\in\HH_r,
\]
be the Fourier expansion of~$\theta_{\mathrm{KM},r}(\varphi)$.
The Fourier coefficients indexed by~$T\geq0$ are Poincaré dual to the special cycles~$Z(T,\varphi)$, in the sense that
\begin{align*}
\int_{X_K} c_T(v,\varphi) \wedge \eta &= e^{-2\pi \tr Tv}\int_{Z(T,\varphi)} \eta
\\
&=
e^{-2\pi \tr Tv}
(-1)^{r-\mathrm{rk}(T)}
\sum_{h\in H_x(\A_f)\backslash H(\A_f)/K}\varphi(h^{-1}x)\int_{Z(x,h)}\Omega^{r-\mathrm{rk}(T)}\wedge \eta
\end{align*}
for every \textcolor{\newcol}{closed} compactly supported $2(n-r)$-form~$\eta$.
By~\cite[Theorem~2.1]{kudlamillson-tubes}, see also~\cite[Theorem~4.20]{kudla-integrals}, the previous equalities are satisfied also for \emph{bounded} $\eta$, e.g.~$\eta=\Omega^{n-r}$, not necessarily of compact support.

These observations and Theorem \ref{KM} imply that
\begin{equation}\label{eq:scugsofc}
    \operatorname{deg}\Theta_r^{\mathrm{coh}}(\tau,\varphi)=
    \int_{X_K}\theta_{\mathrm{KM},r}(\tau,\varphi)\wedge \Omega^{n-r}
    =
    \sum_{T\geq 0}\operatorname{deg}\big(Z(T,\varphi)\big)q^T
\end{equation}
is a holomorphic Siegel modular form of weight $k=1+n/2$ and genus $r$.
As we will see below, the geometric Siegel--Weil formula realizes such generating series of volumes as an Eisenstein series.


We now construct an Eisenstein series similarly as we did in Section~\ref{sec:SiegEisser}, and refer to~\cite[Section~I]{kudla-survey} for details.
The space of the principal series induced representation splits as
\[
I_r(s,\chi_V)=I_{r,\infty}(s,\chi_V)\otimes I_{r,f}(s,\chi_V).
\]
An intertwining map~$\lambda_\infty$ and~$\lambda_f$ can be defined on each factor similarly as in Remark~\ref{rem:intertwining}.

The standard  maximal compact subgroup~$K_\infty'$ of~$\Sp_r(\RR)$ constructed in~\eqref{eq:compmaxofsympl} admits infinitely many characters~$\chi_m$ for $m\in\ZZ$, \textcolor{\newcol}{defined as
\[
\chi_m(k)= \det(\vect{k})^m\qquad \text{for every~$k\in K_\infty'$,}
\]
where the notation is as in~\eqref{eq:compmaxofsympl}.
If~$\varphi_\infty\in S(V(\RR)^r)$ has weight~$m$, i.e.\
\[
\omega_\infty(k)\varphi_\infty=\chi_m(k)\varphi_\infty\qquad \text{for every~$k\in K_\infty'$},
\]
then~$\lambda_\infty(\varphi_\infty)=\varphi_\infty(0)\Phi_m(s_0)$, where~$\Phi_m(s)\in \mathcal{I}_{r,\infty}(s,\chi_\infty)$ is the unique standard section such that~$\Phi_m(k,s)=\chi_m(k)$.}

For $\varphi\in S(V(\A_f)^r)^K$, consider the Eisenstein series for $\operatorname{Sp}_r$ given by
\begin{equation}\label{eq:spseisser}
    E(\tau,s,\lambda_f(\varphi)\otimes \Phi_k)\coloneq \det(v)^{-k/2}E(g_\tau,s,\lambda_f(\varphi)\otimes\Phi_k),
\end{equation}
where $g_\tau$ is the standard element of~$\Sp_r(\RR)$ mapping $i 1_r\in\HH_r$ to $\tau$.
Note that the Eisenstein series~\eqref{eq:spseisser} is an automorphic form, since~$\lambda_f(\varphi)\otimes\Phi_k$ is~$K'$-finite.
In fact,~$\Phi_k$ is certainly~$K_\infty'$-finite, and every~$\varphi\in S(V(\A_f)^r)$ is~$K_f'$-finite. The latter property follows from the fact that the Weil representation~$\omega_f$ is a smooth representation of~$G(\A_f)$.

\begin{Rem}
    Suppose that there exists an even unimodular lattice $L\subset V$.
    For $K=C^0(\hat{L})^\times\cap H(\A_f)$ and $\varphi=\operatorname{char}(\hat{L}^r)\subset S(V(\A_f)^r)$, the Eisenstein series~\eqref{eq:spseisser} boils down to a classical real analytic Siegel Eisenstein series
    \[
        E(\tau,s,\lambda_f(\varphi)\otimes\Phi_k)=\sum_{\gamma\in\Gamma_\infty\backslash \operatorname{Sp}_r(\ZZ)}\det(c\tau+d)^{-k}(\det\operatorname{Im}(\gamma \tau))^{\frac{1}{2}(s+\rho_r-k)},
    \]
    see~\cite[Section~IV.2]{kudla-survey} for details.
\end{Rem}

We are now ready to state the geometric version of the Siegel--Weil formula.
This was first proved by Kudla for~$r=1$ in~\cite{kudla-integrals}.
As explained in~\cite[Section~4]{kudla-survey-speccy}, the same proof generalizes to higher codimension.
\begin{teor}[Geometric Siegel--Weil Formula]\label{thm:geomSWF}
Let $\varphi\in S(V(\A_f)^r)^K$ and let $k=1+n/2$.
    If $r<\frac{n-1}{2}$, then
    \begin{equation}\label{eq:GSWF}
        \operatorname{deg}\Theta_r^{\mathrm{coh}}(\tau,\varphi)=
        (-1)^r\operatorname{vol}(X_K,\Omega^n)E(\tau,s_0,\lambda_f(\varphi)\otimes\Phi_k),
    \end{equation}
    where $s_0=\frac{n-r+1}{2}$.
\end{teor}
Note that the assumption~$r<\frac{n-1}{2}$ implies that the Weil convergence criterion holds.
It is possible to state a version of Theorem~\ref{thm:geomSWF} assuming only the Weil convergence criterion, replacing the left-hand side of~\eqref{eq:GSWF} with the generating series of volumes of special cycles, i.e.,~the right-hand side of~\eqref{eq:scugsofc}, see~\cite[Theorem~4.1]{kudla-survey-speccy}.

The adelic Siegel--Weil formula considered in Section~\ref{sec:adverSWF} was stated for indefinite quadratic spaces~$V$ of signature~$(b^+,b^-)$, while in the present section we work with spaces of signature~$(n,2)$.
For general signature, the symmetric space associated to the orthogonal group of $V$ is not Hermitian in general, hence there is no Shimura variety arising from it.
However, one can work in the framework of \emph{locally symmetric spaces of orthogonal type}, as in~\cite{kudlamillson}, and define special cycles in this setting, replacing the Hodge class~$[\mathcal{L}]$ by an Euler class.
By~\cite{kudlamillson} the generating series of cohomology classes of special cycles is still a modular form, and for~$b^-$ odd it is cuspidal.
If~$b^+$ is even, the generating series of volumes (computed with respect to some Euler class) is still an Eisenstein series.
In this sense, Theorem~\ref{thm:geomSWF} is a special case of a more general statement.

\begin{proof}[Proof of Theorem~\ref{thm:geomSWF}]
    The idea is to apply the classical Siegel--Weil formula of Section~\ref{sec:adverSWF} to some theta integral.
    To do so, we first need to rewrite the integral $\deg\Theta_r^{\mathrm{coh}}(\tau,\varphi)=\int_{X_K}\theta_{\mathrm{KM},r}(\tau,z,h,\varphi)\wedge\Omega^{n-r}$ in terms of some theta \emph{function} on the group~$H$ rather than a theta \emph{form} on~$D$.

    Since $\varphi_{\mathrm{KM}}^r\wedge\Omega^{n-r}$ is a Schwartz function on~$V(\RR)^r$ with values in top-degree differential forms on~$X_K$, there exists some~$\widetilde{\varphi}_{\mathrm{KM}}^r\in S(V(\RR)^r)\otimes A^0(D)$ such that
    \[
    \varphi_{\mathrm{KM}}^r\wedge\Omega^{n-r}=\widetilde{\varphi}_{\mathrm{KM}}^r\wedge\Omega^n.
    \]
    The Schwartz function~$\widetilde{\varphi}_{\mathrm{KM}}^r$ is invariant with respect to~$H(\RR)$, in the sense that
    \[
    \widetilde{\varphi}_{\mathrm{KM}}^r(hx,hz)=\widetilde{\varphi}_{\mathrm{KM}}^r(x,z)\qquad\text{for all $h\in H(\RR)$.}
    \]
    Furthermore, it is of weight~$k$ with respect to the standard maximal compact subgroup of~$\Sp_r(\RR)$ and~$\lambda_\infty(\widetilde{\varphi}_{\mathrm{KM}}^r(\cdot{,}z))=(-1)^r\Phi_k$ for all~$z\in D$ \textcolor{\newcol}{by~\cite[Theorem~3.1]{kudlamillsonI}}.
    By the Siegel--Weil formula (Theorem~\ref{thm;SiegelweilKR}), we then deduce that
    \begin{equation}\label{eq:avinttildevkm}
    \det (v)^{-k/2}I(g_\tau,\varphi\otimes \widetilde{\varphi}_{\mathrm{KM}}^r)
    =(-1)^r E(\tau,s_0,\lambda_f(\varphi)\otimes\Phi_k).
    \end{equation}

    It remains to relate the integral defining~$\deg\Theta_r^{\mathrm{coh}}(\tau,\varphi)$ to~\eqref{eq:avinttildevkm}.
    In fact, they agree up to some constant factor $C_K$ that depends on~$K$, namely
    \begin{equation}\label{eq:inprGSWF}
    \begin{split}
        \deg\Theta_r^{\mathrm{coh}}(\tau,\varphi)&=\int_{X_K}\theta_{\mathrm{KM},r}(\tau,z,h,\varphi)\wedge\Omega^{n-r}
        \\
        &=
        C_K\int_{\operatorname{SO}(V)(\QQ)\backslash\operatorname{SO}(V)(\A)}\theta(\tau,h_\infty h_f,\varphi\widetilde{\varphi}_{\mathrm{KM}}^r)dh,
        \end{split}
    \end{equation}
    where~$dh$ is the Haar measure on~$\operatorname{SO}(V)(\A)$ giving volume one to $\operatorname{SO}(V)(\QQ)\backslash\operatorname{SO}(V)(\A)$.
    Here we split~$dh=dh_\infty\otimes dh_f$ in Archimedean and finite measures, where~$dh_\infty$ is the measure induced by~$\Omega^n$.
    To conclude the proof, one computes that~$C_K=\mathrm{vol}(X_K)$ from the normalization $\int_{\operatorname{SO}(V)(\QQ)\backslash\operatorname{SO}(V)(\A)}dh=1$ by an unfolding argument, see~\cite[Remark~4.18]{kudla-integrals}.
\end{proof}

\begin{Rem}
    If $X_K$ is not compact and the Weil convergence criterion is not satisfied, namely $r\geq n+1-\mathrm{Witt}(V)$, then $E(\tau,s_0,\lambda_f(\varphi_f)\otimes\Phi_k)$ is non-holomorphic with respect to~$\tau$ in general.
\end{Rem}
\begin{Exa}
    Let $V$ be as in Example~\ref{ex:Hilbertmodsurf}, and let~$K$ be such that~$X_K$ is compact.
    Let~$r=1$, and let~$\varphi$ be the characteristic function of some even lattice~$L\subset V$.
    Then $Z(T,\varphi)$ is (a compact analogue of) a Hirzebruch--Zagier divisor. We then have
    \[
        \operatorname{deg}\Theta_1^{\mathrm{coh}}(\tau,\varphi)=\sum_{T\geq 0}\deg Z(T,\varphi)q^T=E(\tau,1,\lambda_f(\varphi)\otimes \Phi_2),
    \]
    where the last term is the weight $ 2$ Eisenstein series for $\Gamma_0(N)$ and character $\chi_\Delta$, with the same notation as Theorem~\ref{thm:thetatrfor}.
    This can be extended to (non-compact) Hilbert modular surfaces with a careful study of the intersection numbers of special divisors in some toroidal compactification of~$X_K$.
\end{Exa}

\section{The arithmetic Siegel--Weil formula}

We keep the basic notation introduced in Section \ref{sect:5.1}.
Kudla initiated a program to link the arithmetic geometry of special cycles on \emph{integral models} of  
orthogonal (and unitary) Shimura varieties to Siegel (and Hermitian)  modular forms, see e.g.~\cite{Kudla-Annals}, \cite{Kudla-MSRI}, and \cite{KRY-book}. In this section we give a short introduction to this circle of ideas focussing on generating series and the Siegel--Weil formula. For further connections to automorphic~$L$-functions and representations we also refer to the recent surveys of Chao Li \textcolor{\newcol}{\cite{li-fromsum2squares},} \cite{Li-survey}.

Kudla conjectured (and in certain cases it is known by now) that there should be results which parallel those of the previous section.
In particular, there should be a modularity result as Theorem~\ref{KM} and an arithmetic analogue of the geometric Siegel--Weil formula Theorem~\ref{thm:geomSWF}.
Since the integral models are defined over an affine base scheme (typically~$\operatorname{Spec}(\ZZ)$) one needs a replacement for cohomology and intersection theory that works in this setting. Kudla proposed to work with the arithmetic Chow theory developed by Gillet and Soul\'e (see e.g.~\cite{GilletSoule} and \cite{Soule-book}) as a higher dimensional generalization of Arakelov geometry for arithmetic surfaces.
So we would like to replace $X_K$ by an integral model~$\mathcal{X}_K$ over $\operatorname{Spec}(\ZZ)$, the special cycles $Z(T,\varphi)$ by arithmetic cycles $\widehat{\mathcal{Z}}(T,\varphi)$ in the sense of Gillet--Soul\'e, and instead of studying classes in cohomology $H^{2r}(X_K,\CC)$ we consider classes in arithmetic Chow groups.

We now explain this in somewhat more detail, focussing on the main ideas and glossing over some technical details.
%
%
We assume that $L\subset V$ is a maximal even lattice of square-free discriminant and consider the compact open subgroup  $K=H(\A_f)\cap C^0(\widehat{L})^\times$ of $H(\A_f)$.
By work of Kisin, Vasiu, and Madapusi, the Shimura variety $X_K$ has a canonical integral model $\mathcal{X}_K$, which is a regular flat stack over $\operatorname{Spec}\ZZ$,  see \cite{Kisin-integral}, \cite{Madapusi-Toroidal}. Moreover, there exist toroidal compactifications, which are regular, flat, and projective over $\operatorname{Spec}(\ZZ)$. We fix one of these throughout and denote it by $\overline{\mathcal{X}}_K$.
For a Schwartz--Bruhat function 
\[
    \textcolor{\newcol}{\varphi\in S_{L,r}\coloneqq\bigoplus_{\mu\in(L'/L)^r}\CC\operatorname{char}(\mu+\widehat{L}^r)}
\] 
and $T\in \mathrm{Sym}_r(\QQ)$, 
Howard and Madapusi constructed in \cite{HM-Compositio} integral models~$\mathcal{Z}(T,\varphi)$ of the special cycles~$Z(T,\varphi)$ \textcolor{\newcol}{using a moduli description}. Moreover, Garcia and Sankaran \cite{GS18} constructed Green currents $G(T,v,\varphi)$ for these cycles. They depend on a parameter $v$ in the positive definite symmetric real $r\times r$-matrices, which should be  thought of as the imaginary part of a variable $\tau =u+iv\in \HH_r$. The current $G(T,v,\varphi)$ is given by a smooth differential form on~$X_K\setminus Z(T,\varphi)$ with a logarithmic singularity along the cycle. It is a representative for the cohomology class of the cycle, more precisely, it satisfies the identity of currents  
\[
dd^c[G(T,v,\varphi)] +\delta_{Z(T,\varphi)} = [dd^c G(T,v,\varphi)]
\]
on smooth test forms with compact support. Here the current on the right hand-side is given by a smooth form, essentially by the Poincar\'e dual form obtained by suitably averaging over the Kudla--Millson forms. The pair 
\[
\widehat{\mathcal{Z}}(T,v,\varphi) = \left(\mathcal{Z}(T,\varphi), G(T,v,\varphi)\right) 
\]
determines an arithmetic cycle in the sense of Gillet--Soul\'e.
We are interested in the classes of these cycles in the arithmetic Chow group  
$\widehat{\mathrm{Ch}}_{\CC}^{r}(\mathcal{X}_K)$.
\textcolor{\newcol}{It turns out that even for positive definite $T$, the cycles $\mathcal{Z}(T,\varphi)$ can 
contain irreducible components in the wrong codimension in finite characteristics. This problem can be resolved by working with a derived tensor product construction to define corrected special cycle classes in the Chow group of degree $r$, see \cite[Section~5]{HM-Compositio}.}
Note that these classes can be defined for $r$ in the range $1\leq r \leq n+1$, since the \emph{total} dimension of $\mathcal{X}_K$ is $n+1$.

\begin{conj}[Kudla]
    The generating series
    \[
        \Theta_r^{\mathrm{arith}}(\tau,\varphi)=\sum_T [\widehat{\mathcal{Z}}(T,v,\varphi)]q^T,\quad \tau=u+iv\in \HH_r,
    \]
    \textcolor{\newcol}{of the corrected special cycle classes}
    is a (non-holomorphic) Siegel modular form of weight $k=1+\frac{n}{2}$ \ for $\operatorname{Sp}_r$ with values in~$\widehat{\mathrm{CH}}^r_\CC(\mathcal{X}_K)$.
\end{conj}
If $r=1$ (and $n$ arbitrary) the conjecture  can be proved by combining work of Howard--Madapusi \cite{HM-Asteristque} on special divisors and Borcherds products on the integral model $\mathcal{X}_K$ with results by Ehlen--Sankaran \cite{ES18} and \cite{BF04} on  Green functions for special divisors.
For~$r\geq 2$ it is only known in a few cases in general. In the special case of Shimura curves (where $n=1$), complete results are obtained for~$r=2$ by Kudla--Rapoport--Yang in \cite{KRY-book}. In top degree $r=n+1$ the arithmetic Siegel--Weil formula below provides some further evidence.

The analogous problem can also be studied on the toroidal  compactification  $\overline{\mathcal{X}}_K$. This first requires an extension of the $\widehat{\mathcal{Z}}(T,v,\varphi)$ to the compactification, which involves the addition of  suitable cycles supported at the boundary. It also requires a careful study of the growth behavior of the Green currents at the boundary. In this setting it turns out that one has to work with an extension of arithmetic Chow theory due to Burgos--Kramer--K\"uhn \cite{BKK} which allows for additional  mild singularities of Green currents at the boundary. For $r=1$ such extensions of the arithmetic divisors have been constructed in \cite{BZ-toroidal} and a modularity result follows by combining this with \cite{HM-Asteristque}.
In the similar case of arithmetic special divisors on unitary Shimura varieties of signature $(n,1)$ it was proved in \cite{BHKRY} over imaginary quadratic fields and in \cite{Qiu_Xu_2025} over larger CM fields.

The composition of the natural pushforward map $\widehat{\mathrm{CH}}^{n+1}(\overline{\mathcal{X}}_K)\to \widehat{\mathrm{CH}}^{1}(\operatorname{Spec}(\ZZ))$ with the arithmetic degree map $\widehat{\mathrm{CH}}^{1}(\operatorname{Spec}(\ZZ))\to \RR$ induces an arithmetic degree map
\[
    \widehat{\mathrm{deg}}\colon \widehat{\mathrm{CH}}_\CC^{n+1}(\overline{\mathcal{X}}_K)\rightarrow \CC.
\]
It turns out that  (for regular $T$), in this top degree situation, the special cycles $\widehat{\mathcal{Z}}(T,v,\varphi)$ determine classes on the compactification $\overline{\mathcal{X}}_K$. This naturally leads to  the question: What is the image of the above generating series $\Theta_{n+1}^{\mathrm{arith}}(\tau,\varphi)$?

\begin{conj}[Kudla]
\label{conj:arithsw}
The generating series of arithmetic  special cycles of top degree is given by 
    \begin{align*}
        \widehat{\deg}\ \Theta_{n+1}^{\mathrm{arith}}(\tau,\varphi)&=\sum_{T\in\operatorname{Sym}_{n+1}(\QQ)}\widehat{\deg}\ (\widehat{Z}(T,v,\varphi))q^T\\
        &=C\cdot E'(\tau,0,\lambda_f(\varphi)\otimes \Phi_k),
    \end{align*}
    that is, the central derivative of a Siegel Eisenstein series of genus $r=n+1$ and weight $k$ associated with $\varphi$.
Here, $\tau=u+iv\in \HH_{n+1}$, and $C$ is a certain normalizing constant.
\end{conj}

This conjecture can be viewed as an arithmetic variant of the Siegel--Weil formula, analogous to the geometric Siegel--Weil formula in Theorem \ref{thm:geomSWF}. Here $E(\tau,s,\lambda_f(\varphi)\otimes \Phi_k)$ denotes the Siegel Eisenstein series as in~\eqref{eq:eis} in genus $n+1$ viewed as a function on~$\HH_{n+1}$ as in \eqref{eq:spseisser}.
   It is \emph{incoherent}  in the sense of Kudla, since 
   at the archimedean place it involves the standard section of weight $k$, which is obtained by applying the intertwining operator $\lambda_\infty$ to the Gaussian on the standard positive definite real quadratic space $\RR^{n+2,0}$ of dimension~$n+2$ instead of on the indefinite space $V_\RR\cong \RR^{n,2}$ of signature $(n,2)$. Note that the Hasse invariants of these two real quadratic spaces differ by the sign, while their dimension and quadratic character agree. Hence the Eisenstein series is associated to the incoherent collection $(V_p)_{p\leq \infty}$ of local quadratic spaces given by $V_p=V(\QQ_p)$ for primes~$p<\infty$ and~$V_\infty=\RR^{n+2,0}$. This collection cannot be obtained by taking the base change to~$\QQ_p$ for~$p\leq \infty$ from a single global quadratic space over~$\QQ$, because of the product formula for the Hasse invariant.
   
   The incoherence of the above Eisenstein series implies that it satisfies an odd functional equation under $s\mapsto -s$ and therefore vanishes at the center of symmetry $s=0$. Its derivative with respect to~$s$ at this central point enters in the above formula. Note that~$s_0=\dim(V)/2-\rho_{r}=0$ is exactly the Siegel--Weil point in this case.

In the case $n=0$, where $\mathcal{X}_K$ is a moduli stack over $\ZZ$ of elliptic curves with complex multiplication, the conjecture was proved in \cite{KRY-tiny}, in the $n=1$ case of Shimura curves in \cite{KRY-book}.
In general, one can study the identity of Conjecture \ref{conj:arithsw} coefficientwise and break it up into local identities, called \emph{local} arithmetic Siegel--Weil formulas. For regular non-positive definite $T$, this was first shown in \cite{BY-JEMS}, based on a proof of the archimedean arithmetic local Siegel--Weil formula. 
For positive definite $T$ it was proved in \cite{Li-Zhang-ASW}, where a key contribution was the non-archimedean local Siegel Weil formula.  
A different proof of the archimedean local Siegel--Weil formula was given by Garcia--Sankaran \cite{GS18}. It relies on their construction of Green currents using  Quillen's work on super connections for Hermitian vector bundles \cite{Qu85}, which works in great generality. In particular, it also works for $T$ which are not necessarily regular. However, global results for singular $T$ are only available in rather special cases at present.
In the similar case of unitary Shimura varieties of signature $(n,1)$ analogous results were obtained in \cite{Liu-ANT1}, \cite{Li-Zhang-JAMS}.

\printbibliography

\end{document}